\documentclass[12pt]{amsart}
\usepackage{amsmath}
\usepackage{amssymb}
\usepackage{amscd}
\usepackage{multicol}
\usepackage{MnSymbol}
\usepackage{tikz}
\usepackage[OT2,T1]{fontenc}
\usepackage{mathrsfs}
\usepackage{comment} 
\usepackage{stmaryrd}
\usepackage[all]{xy}
\usepackage{longtable}
\usepackage{multirow,bigdelim}
\DeclareSymbolFont{cyrletters}{OT2}{wncyr}{m}{n}
\DeclareMathSymbol{\Sha}{\mathalpha}{cyrletters}{"58}

\title[Thakur hypergeometric functions]
{On the period interpretation for some special values of Thakur hypergeometric functions}
\author{Ryotaro Harada}
		\address{
				University of the Ryukyus, Senbaru 1, Nishihara-cho, Nakagami-gun, Okinawa 9010213, Japan.}
			\email{harada@math.u-ryukyu.ac.jp}
\date{May 2, 2022}

\newtheorem{thm}{Theorem}[section]

\newtheorem{prop}[thm]{Proposition}
\theoremstyle{remark}
        
\makeatletter
\@namedef{subjclassname@2020}{%
  \textup{2020} Mathematics Subject Classification}
\makeatother
\theoremstyle{definition}
\newtheorem{defn}[thm]{Definition}
\newtheorem{rem}[thm]{Remark}

\newtheorem{eg}[thm]{Example}

\makeatletter

 \DeclareMathOperator{\GL}{GL}
\DeclareMathOperator{\Mat}{Mat}

 \DeclareMathOperator{\wt}{wt}

\DeclareMathOperator{\dep}{dep}

\begin{document}
\bibliographystyle{amsalpha+}

\begin{abstract}
In 1995, Thakur invented and studied positive characteristic analogues of hypergeometric functions. In this paper, we interpret the special values of those functions as periods of a pre-$t$-motive. As a consequence, we show their transcendence and linear independence results by using Chang's refined version of the Anderson--Brownawell--Papanikolas criterion. Furthermore, as by-products, we show some linear/algebraic independence results among the special values of Kochubei multiple polylogarithms according to our period interpretation and the corresponding refined criterion.
\end{abstract}


\maketitle
\tableofcontents
\setcounter{section}{-1}
\section{Introduction}
\subsection{Thakur hypergeometric functions}
Let $\mathbb{F}_q$ be a fixed finite field with $q$ elements, where $q$ is the power of prime number $p$.
    Let $\mathbb{P}^1$ be a projective line defined over $\mathbb{F}_q$ with a fixed point at infinity $\infty\in\mathbb{P}^1(\mathbb{F}_q)$.
    Let $A$ be the ring of regular functions away from $\infty$, with $k$ as its fraction field. Let $k_\infty$ be the completion of $k$ at $\infty$, and let $\mathbb{C}_\infty$ be the completion of a fixed algebraic closure of $k_\infty$. With the variable $\theta$, we can identify $A$ with the polynomial ring $\mathbb{F}_q[\theta]$ and $k$ with the rational function field $\mathbb{F}_q(\theta)$.
Thakur defined and studied the positive characteristic analogues of the classical hypergeometric functions (HGFs) in \cite{T95}. His definition is motivated by Barns integral representation (\cite[Section 3.4]{T95}) of the HGFs. For $D_i:=\prod^{i}_{j=1}(\theta^{q^j}-\theta)^{q^{i-j}}\ (D_0:=1)$ and $L_i:=\prod^{i}_{j=1}(\theta-\theta^{q^j})\ (L_0:=1)$ with $i\in\mathbb{Z}_{\geq 0}$, 
he found the following analogue of the Pochhammer symbols:
\[
	(a)_n:=\begin{cases}
	   		D_{n+a-1}^{q^{-(a-1)}}\quad &\text{if $a\geq 1$},\\	
	   		1/L_{-a-n}^{q^n}\quad &\text{if $0\geq a$ and $-a\geq n$}, \\
	   		0\quad &\text{if $n> -a\geq 0$}.
	   \end{cases} 
\]
Then, based on the characteristic 0 case, he defined the following analogue of HGFs by using the above symbols:
for $a_1, \ldots, a_r\in\mathbb{Z}$ and $b_1, \ldots, b_s\in\mathbb{Z}_{>0}$,
\begin{align}\label{s0hyper}
_rF_s(z):=~_rF_s(a_1, \ldots, a_r;b_1, \ldots, b_s)(z):=\sum_{n\geq 0}\frac{(a_1)_n\cdots (a_r)_n}{(b_1)_n\cdots (b_s)_nD_n}z^{q^n}\in\mathbb{C}_{\infty}\llbracket z\rrbracket.
\end{align}
Throughout this paper, we call these functions the Thakur hypergeometric functions (THGFs) and without loss of generality we assume that $a_{i}\leq a_{j}$ and $b_{i}\leq b_{j}$ for $i\leq j$.
The THGFs $_rF_s(z)$ have three cases of the convergence domain as follows: 
\[
	_rF_s(z) \text{\ are defined for\ }\begin{cases}     
										z=0\quad &\text{\ if $r>s+1$},\\
										z\in\mathbb{C}_{\infty}\quad &\text{\ if $r<s+1$},\\
										z\in\mathbb{C}_{\infty} \text{\ with\ } |z|_{\infty}<q^{\sum_{j=1}^s(b_j-1)-\sum_{j=1}^r(a_j-1)} &\text{\ if $r=s+1$}.
                                       \end{cases}
\] 
Here, $| \cdot |_{\infty}$ is the absolute value for $\mathbb{C}_{\infty}$ such that $|\theta|_{\infty}=q$.
Thakur also showed that THGFs satisfy an analogue of the hypergeometric differential equation by using the Carlitz differential operator $\Delta_a$ and the Carlitz derivative $d_F$. The case of $\Delta_0$ is originated in \cite{Ca35}. They are $\mathbb{F}_q$-linear operators defined for $\mathbb{F}_q$-linear functions $f(z)$ by
\begin{align*}
 \Delta_a\bigr(f(z)\bigl):=f(\theta z)-\theta^{q^{-a}}f(z) \text{\ for $a\in\mathbb{Z}$},\quad d_F\bigl(f(z)\bigr):=\Delta_0\bigl(f(z)\bigr)^{1/q}.
\end{align*}
 
We can consider operator $\Delta_a$ and $d_F$ to be the positive characteristic analogue of $z(d/dz)+a$ and $d/dz$ respectively. For more details and studies, see \cite[\S 3.1]{T95} and \cite{T00}. 

Then, Thakur demonstrated the following differential equation (\cite[(10)]{T95}) which is seen to be an analogue of the hypergeometric differential equation:
\[
	d_F\circ\Delta_{a_1}\circ\cdots\circ\Delta_{a_r}\bigl(~_rF_s(z)\bigr)=\Delta_{b_1-1}\circ\cdots\circ\Delta_{b_s-1}\bigl(~_rF_s(z)\bigr).
\]
Furthermore, he discovered several properties of $~_rF_s(z)$, including the analogue of contiguous relations, the summation formula, specializations to exponential functions, the Bessel functions, the Jacobi and Legendre polynomials in positive characteristic, and the connection to the tensor powers of the Carlitz modules in \cite{T95}.  

The second analogue of the HGF was also defined in \cite{T95} by using the positive characteristic analogue of the binomial coefficients. In this paper, we discuss only the first analogue recalled in \eqref{s0hyper}.

Later, Thakur, Wen, Yao and Zhao (\cite{TWYZ}) obtained a sufficient condition for the special values of $_rF_s(z)\ (r<s+1)$ and an equivalent condition for the special values of $_rF_s(z)\ (r=s+1)$ to be transcendental over $k$. These transcendence results were the consequence of their Diophantine criterion for transcendence in positive characteristic, which generalized Theorem 1 of \cite{Y06}. Moreover, there are some approaches for solving transcendence/linear independence problems via periods (see Definition \ref{defperiod}), as developed by Anderson--Brownawell--Papanikolas \cite{ABP04}, namely, the so-called ABP criterion. For example, Carlitz multiple polylogarithms are firstly defined in \cite{C14} as follows and their special values are known as periods: for $\mathfrak{s}=(s_1, \ldots, s_r)\in\mathbb{Z}_{>0}^r$ and ${\bf z}=(z_1, \ldots, z_r)\in\{{\bf z}\in\mathbb{C}^r_{\infty}\ |\ |z_1/\theta^{\frac{qs_1}{q-1}}|^{q^{i_1}}_{\infty}\cdots|z_r/\theta^{\frac{qs_r}{q-1}}|^{q^{i_r}}_{\infty}\rightarrow 0\ \text{as}\ 0\leq i_r<\cdots<i_1\rightarrow \infty \}$,
\[
	Li_{C, \mathfrak{s}}({\bf z}):=\sum_{i_1>\cdots>i_r\geq0}\frac{z_1^{q^{i_1}}\cdots z_r^{q^{i_r}}}{L_{i_1}^{s_1}\cdots L_{i_r}^{s_r}}\in\mathbb{C}_{\infty}.
\]
For each index $\mathfrak{s}\in\mathbb{Z}_{>0}^r\ (r>0)$, we set the weight as $\wt(\mathfrak{s}):=s_1+\cdots+s_r$ and the depth as $\dep(\mathfrak{s}):=r$. Analogous to the classical case, the Carlitz multiple polylogarithms $Li_{C, \mathfrak{s}}({\bf z})$ include the Carlitz polylogarithms introduced in \cite{AT90} and the Carlitz logarithms introduced in \cite{Ca35} as $\dep(\mathfrak{s})=1$ cases and $\dep(\mathfrak{s})=\wt(\mathfrak{s})=1$ cases respectively. The ABP criterion is applied to the linear independence of the Carlitz multiple polylogarithms at algebraic points with different weights in \cite{C14}, and the criterion for Eulerian Carlitz multiple polylogarithms at algebraic points in \cite{CPY19}.

Our motivation in this paper is to develop a period interpretation for the special values of $_{s+1}F_s(z)$ with some $q^d$-th powers and to provide linear independence results among these values by using Chang's refined version of ABP criterion (see Theorem \ref{changcri}). Moreover, as by-products, we present the linear/algebraic independence results for the special values of Kochubei multiple polylogarithms (KMPLs).
\subsection{Main results}
In the characteristic 0 case, it is known that the HGFs can be given by periods of algebraic varieties in some cases, particularly with rational parameters. 
For example, the special values of $\pi_2F_1(1/2, 1/2; 1)(z)$ and $\pi\sqrt{-1} _2F_1(1/2, 1/2; 1)(1-z)$ are known to be periods related to the elliptic curve $y^2=x(x-1)(x-z)$ (cf. \cite[\S 2.2]{KZ01}).  
For the $_{s+1}F_s$ case, see the survey of hypergeometric motives \cite{RRV}. There is also the study about the transcendence for values of HGFs. For example, Schwarz determined the list of HGFs of $_2F_1$ cases, which are algebraic functions in \cite{S} later generalized to the $_{s+1}F_s$ case by Beukers and Heckman (\cite{BH88}). 
Furthermore, there are studies about the linear independence for the special values of HGFs. For the recent works, Fischler and Rivoal (\cite{FR20}) and David, Hirata--Kohno, Kawasaki (\cite{DHKK}) proved the linear independence among the values of HGFs with some different algebraic points or some different rational parameters.

Our main results include a period interpretation of the special values of the THGFs (Theorem \ref{periodhgs}), the transcendence and linear independence results of some special values of the THGFs (Theorems \ref{transthm}, \ref{lindepdifferentn} and \ref{lindepdiffalpha}) and the linear/algebraic independence results of the KMPLs at algebraic points (Theorems \ref{lininddkmpl}--\ref{linindkplcpl}). We note that our Theorem \ref{transthm} also gives equivalent conditions for the special values of $_{s+1}F_s(z)$ to be transcendental over $k$, which is the same as the result in \cite{TWYZ} but the proof is different. Indeed we use a period interpretation for the special values of the THGFs.   

We can show that the $q$-th power of the THGFs at algebraic points appears as periods of a specific pre-$t$-motive $M_{{\bf a. b}, d}$ defined by \eqref{hgfmotive}.
We define 
\[
	\Omega:=(-\theta)^{\frac{-q}{q-1}}\prod^{\infty}_{i=1}\Bigl(1-\frac{t}{\theta^{q^i}}\Bigr)\in k_{\infty}(\theta^{\frac{1}{q-1}})\llbracket t \rrbracket
\]
where we fix a $(q-1)$-th root of $-\theta$. The Carlitz period $\tilde{\pi}$ is defined by $(\Omega|_{t=\theta})^{-1}$.

Let
$d\geq\max\{b_1, b_2, \ldots, b_s\}$ for given $b_1, b_2, \ldots, b_s\in\mathbb{Z}_{>0}$. We further define the following power series:
\begin{align}\label{qdef}
	P_{{\bf b},d}&:=(-\theta)^{\frac{-\sum_{j=1}^s(b_j-1)q^{d-1}}{q-1}}\prod_{l=1}^{\infty}\prod_{\substack{j=1\\ b_j\geq 2}}^s
	\Biggl\{\biggl(1-\frac{t}{\theta^{q^l}} \biggr)^{q^{b_j-2}}\biggl(1-\frac{t}{\theta^{q^{l+1}}} \biggr)^{q^{b_j-3}}\cdots\biggl(1-\frac{t}{\theta^{q^{l+b_j-2}}} \biggr) \Biggr\}^{q^{d-b_j}}
	\end{align}
	\begin{align*}
	&\hspace{-6cm}=\prod_{j=1}^s\frac{1}{\mathbb{D}_{b_j-2}^{q^{d-b_j}}}\prod^{b_j}_{i=2}\Omega^{q^{d-i}}\in k_{\infty}(\theta^{\frac{1}{q-1}})\llbracket t \rrbracket. 
\end{align*}
\begin{rem}
When all $b_i\ (i=1, 2, \ldots, s)$ and $d$ are equal to 2, we obtain $P_{d, {\bf b}}=\Omega^s$.
\end{rem}
Then, our period interpretation for the special values of the THGFs is stated as follows:
\begin{thm}\label{introperiodint}
Let $a_i, b_j\in\mathbb{Z}_{>0}\quad (1\leq i\leq s+1,\ 1\leq j\leq s)$.
Then, for $\alpha\in\overline{k}$ with $|\alpha|_{\infty}<q^{\sum_{j=1}^s(b_j-1)-\sum_{i=1}^{s+1}(a_i-1)}$ and $d\geq\max_{i, j}\{a_i, b_j\}$, the value $_{s+1}F_s(a_1, \ldots, a_{s+1}; b_1, \ldots, b_s)(\alpha)^{q^{d-1}}$ multiplied by $\prod_{j=1}^sD_{b_j-2}^{q^{d-b_j}}\prod_{i=2}^{b_j}\tilde{\pi}^{q^{d-i}}$ is a period of the pre-$t$-motive $M_{{\bf a}, {\bf b}}$.
\end{thm}

\begin{rem}
We also give a period interpretation of the special values of the THGFs without $\prod_{j=1}^sD_{b_j-2}^{q^{d-b_j}}\prod_{i=2}^{b_j}\tilde{\pi}^{q^{d-i}}$ in Remark \ref{simperintthgf}.
\end{rem}

According to our period interpretation and the refined ABP criterion, we obtain the following linear independence results, each of which is restated later in Theorem \ref{transthm}, \ref{lindepdifferentn} and \ref{lindepdiffalpha}. 

\begin{thm}\label{introlinindthgf}
We denote all $m$ satisfying $d\geq m\geq 0$ by $m_i\ (i=1, \ldots, n)$ where $d=\max_{\substack{1\leq i\leq s+1\\ 1\leq j\leq s}}\{a_i, b_j\}$.
\begin{itemize}
\item	We set ${\bf a}=(a_{1}, \ldots, a_{s+1})\in\mathbb{Z}_{>0}^{s+1}$, ${\bf b}=(b_{1}, \ldots, b_{s})\in\mathbb{Z}_{>0}^{s}$ and $\alpha\in\overline{k}^{\times}$ with $|\alpha|_{\infty}<q^{\sum_{j=1}^s(b_j-1)-\sum_{i=1}^{s+1}(a_i-1)}$. Then, $_{s+1}F_s(a_1, \ldots, a_{s+1};b_1, \ldots, b_s)(\alpha)$ is transcendental over $k$ if and only if $b_j>a_{j+1}$ for some $j$.
\item	For any ${\bf a}_s=(a_{1}, \ldots, a_{s+1})\in\mathbb{Z}_{>0}^{s+1}$, ${\bf b}_s=(b_{1}, \ldots, b_{s})\in\mathbb{Z}_{>0}^{s}$ such that $b_1>a_{s+1}$, we take $m_r$ satisfying $b_1-1\geq m_r$. Let ${\bf a}_h=(a_{1}, \ldots, a_{h+1})$, ${\bf b}_h=(b_{1}, \ldots, b_{h})\ (h=1, \ldots, s)$ and $\alpha_h\in\overline{k}^{\times}$ with $|\alpha_h|_{\infty}<q^{\sum_{j=1}^h(b_j-1)-\sum_{i=1}^{h+1}(a_i-1)}$. If $\min_{1\leq i\leq n, i\neq r}\{c(m_i)q^{d-m_i}\}>c(m_r)q^{d-m_r}$, then $_{h+1}F_{h}({\bf a}_h; {\bf b}_h)(\alpha_h)\ (1\leq h\leq s)$ are $\overline{k}$-linearly independent.
\item  For any ${\bf a}=(a_{1}, \ldots a_{s+1})\in\mathbb{Z}_{>0}^{s+1}$ and ${\bf b}=(b_{1}, \ldots b_{s})\in\mathbb{Z}_{>0}^{s}$ such that $b_j>a_{j+1}$ for some $j$ and $\min_{1\leq i\leq n, i\neq u}\{c(m_i)q^{d-m_i}\}>c(m_u)q^{d-m_u}$ for some $u$, let $\alpha_i\in\overline{k}^{\times}\ (i=1,\ldots, r)$ with $|\alpha_i|_{\infty}<q^{\sum_{j=1}^s(b_j-1)-\sum_{i=1}^{s+1}(a_i-1)}$.
If $\alpha_1, \ldots, \alpha_r$ are $k$-linearly independent, then $_{s+1}F_{s}({\bf a}; {\bf b})(\alpha_1), \ldots, _{s+1}F_{s}({\bf a}; {\bf b})(\alpha_r)$ are $\overline{k}$-linearly independent.
\end{itemize}
\end{thm}

For the definition of $c(\cdot)$ and the necessity of its conditions in Theorem \ref{introlinindthgf}, see Definition \eqref{deformthgfexpn} and Remark \ref{cassum}. 
By Proposition \ref{genthgfkpl} and \ref{hgslog}, the above results can be applied to the case of Kochubei polylogarithms (KPLs), which were defined and studied by Kochubei (\cite{Ko05}) as follows:
\[
	Li_{K,s}(z):=\sum_{i\geq 1}\frac{z^{q^i}}{(\theta^{q^i}-\theta)^s}\in\mathbb{C}_{\infty}
\]
for $z\in\mathbb{C}_{\infty}$ with $|z|_{\infty}<q^{s}$. The case of $s=z=1$ was discussed by Wade \cite{Wa41}, who found that $\sum_{i\geq 1}1/(\theta^{q^i}-\theta)$ is transcendental over $k$. The KPLs are considered to be another positive characteristic analogue of the polylogarithms, with a different motivation from that of the 
Carlitz polylogarithms. On the one hand, the Carlitz polylogarithms are generalizations of Carlitz logarithms defined by the formal inverse of the Carlitz exponentials; on the other hand, Kochubei's idea was to obtain the analogue of the classical polylogarithm $Li_s(z):=\sum_{n>0}z^n/n^s$ by finding a function that satisfies the analogue of the differential equation $zd/dzLi_s(z)=Li_{s-1}(z)$. Notably, our definition of the KPLs is given in the $\infty$-adic case, while Kochubei (\cite{Ko05}) defined them in the $v$-adic case ($v$ is a monic irreducible polynomial in $\mathbb{F}_q[\theta]$, a finite place). In \cite{Ko05}, he also defined the analogues of the Riemann zeta values by $\zeta_K(\theta^{-n}):=Li_{K, n}(1)$ which we call the Kochubei zeta values in this paper. 

Based on the following setting, we can define the KMPLs as $\mathfrak{s}:=(s_1, \ldots, s_r)\in\mathbb{Z}_{>0}^r$ and ${\bf z}:=(z_1, \ldots, z_r)\in\mathbb{C}_{\infty}^r$ such that $|z_i|_{\infty}<q^{s_i}$:
\[
	Li_{K, \mathfrak{s}}({\bf z}):=\sum_{i_1>\cdots>i_r>0}\frac{z_1^{i_1}\cdots z_r^{i_r}}{(\theta^{q^{i_1}}-\theta)^{s_1}\cdots(\theta^{q^{i_r}}-\theta)^{s_r}}\in\mathbb{C}_{\infty}.
\]
In this paper, we denote the 1-variable case by $Li_{K, \mathfrak{s}}(z)=\sum_{i_1>\cdots>i_r>0}\frac{z^{i_1}}{(\theta^{q^{i_1}}-\theta)^{s_1}\cdots(\theta^{q^{i_r}}-\theta)^{s_r}}$. 
Similar to the cases of classical multiple polylogarithm (cf.\cite{W02}) and the Carlitz multiple polylogarithm (\cite[\S 5.2]{C14}), the KMPLs also satisfy the sum-shuffle relation by their series expressions. We can describe the relation in the same way of the Carlitz case (\cite[Section 5.2]{C14}). For a given $\mathfrak{s}_1\in\mathbb{Z}^{r_1}_{>0}$ and $\mathfrak{s}_2\in\mathbb{Z}^{r_2}_{>0}$ it is described by
\begin{align}\label{shufflekmpl}
	Li_{K, \mathfrak{s}_1}({\bf z_1})Li_{K, \mathfrak{s}_2}({{\bf z}_2})=\sum_{({{\bf v}_1, {\bf v}_2})}Li_{K, {{\bf v}_1+{\bf v}_2}}({{\bf z}_3}).
\end{align}
Here, ${{\bf v}_1, {\bf v}_2}\in\mathbb{Z}_{\geq 0}^{r_3}$ satisfying ${{\bf v}_1}+{{\bf v}_2}\in\mathbb{Z}_{>0}^{r_3}$ with $\max\{r_1, r_2\}\leq r_3\leq r_1+r_2$ and ${{\bf v}_i}\ (i=1 ,2)$ is obtained by inserting $(r_3-r_i)$ zeros into $\mathfrak{s}_i$ in all possible ways, including in front and the end of $\mathfrak{s}_i$. The pair $({{\bf v}_1, {\bf v}_2})$ runs over all such expressions for all $r_3$ with $\max\{r_1, r_2\}\leq r_3\leq r_1+r_2$. For every such ${{\bf v}_1}+{{\bf v}_2}\in\mathbb{Z}_{>0}^{r_3}$, the $m$th component $z_{3m}$ of ${{\bf z}_3}$ is $z_{in}$ if the $m$th component of ${{\bf v}_i}$ is $s_{in}$, while the $m$th component of ${{\bf v}_j}\ (i\neq j)$ is 0 or $z_{3m}$ is $z_{1n}z_{2l}$ if the $m$th component of both ${{\bf v}_1}$ and ${{\bf v}_2}$ are $z_{1n}$ and $z_{2l}$. 
For example, we have $$Li_{K, s_1}(z_1)Li_{K, s_2}(z_2)=Li_{K, (s_1+s_2)}(z_1z_2)+Li_{K, (s_1, s_2)}(z_1, z_2)+Li_{K, (s_2, s_1)}(z_2, z_1).$$ 
This relation allows us to rewrite the algebraic relations to linear relations among the KMPLs. We can also give a period interpretation of the special values of the KMPLs with a pre-$t$-motive defined by $\Phi_{\mathfrak{s}, {\boldsymbol \alpha}}$ (see \eqref{periodkmpl}), and by applying the ABP criterion, we can obtain linear independence results and an algebraic independence result via the sum-shuffle relation. 

Our linear/algebraic independence results for the special values of KMPLs are stated as follows (each of them is described again later in Theorems \ref{lininddkmpl}-\ref{nonzetalike}):
\begin{thm}\label{intthmkmpl}
\begin{itemize}
\item[(i)] For indices $\mathfrak{s}\in\mathbb{Z}_{>0}^2$ with $\wt(\mathfrak{s})=w$ and $\alpha\in\overline{k}^{\times}$ with $|\alpha|_{\infty}<q^{s_1}$, $Li_{K, \mathfrak{s}}(\alpha)$ are $\overline{k}$-linearly independent.
\item[(ii)] For $s_2\geq s_1>0$, let $\alpha_1, \alpha_2, \alpha_3\in\overline{k}^{\times}$ with $|\alpha_1|_{\infty}, |\alpha_3|_{\infty}<q^{s_1}$ and $|\alpha_2|_{\infty}<q^{s_2}$. Then $Li_{K, (s_1, s_2)}(\alpha_1, \alpha_2)$ and $Li_{K, s_2}(\beta_1)$ are algebraically independent over $\overline{k}$.
\item[(iii)] Let $\alpha, \beta\in\overline{k}^{\times}$ such that $|\alpha|_{\infty}<q^n$, $|\beta|_{\infty}\in q^{nq/(q-1)}$. Then, $Li_{K,n}(\alpha), Li_{C,n}(\beta)$ are $\overline{k}$-linearly independent.
\item[(iv)] Let $\mathfrak{s}=(s_{1}, \ldots, s_{r})\in\mathbb{Z}_{>0}^r$ with $\wt(\mathfrak{s})=w$ and ${\boldsymbol \alpha}=(\alpha_1, \ldots, \alpha_r)\in(\overline{k}^{\times})^r$ with $|\alpha_i|_{\infty}<q^{s_i}\ (i=1, \ldots, r)$, such that $Li_{K,\mathfrak{s}}({\boldsymbol \alpha})\neq 0$. Then $Li_{K,\mathfrak{s}}({\boldsymbol \alpha})$ and $\tilde{\pi}^w$ are $\overline{k}$-linearly independent.
\item[(v)] For $\mathfrak{s}=(s_{1}, \ldots, s_{r})\in\mathbb{Z}_{>0}^r$ such that $\wt(\mathfrak{s})=w$, let ${\boldsymbol \alpha}=(\alpha_1, \ldots, \alpha_r)\in(\overline{k}^{\times})^r$ with $|\alpha_i|_{\infty}<q^{s_i}\ (i=1, \ldots, r)$ and $\beta\in\overline{k}^{\times}$ with $|\beta|_{\infty}<q^w$. Then, $Li_{K, \mathfrak{s}}({\boldsymbol \alpha})$ and $Li_{K, w}(\beta)$ are $\overline{k}$-linearly independent.
\end{itemize}
\end{thm}
We can define Kochubei multizeta values by $\zeta_K(\theta^{-s_1}, \ldots, \theta^{-s_r}):=Li_{K, s_1,\ldots, s_r}(1, \ldots, 1)$. According to the Thakur multizeta values case (\cite{LT14}), we can also define Eulerian/zeta-like indices for Kochubei multizeta values but Theorem \ref{intthmkmpl} (iv) and (v) imply non-existence of such indices. 

Finally, this paper is organized as follows. In Section 1, we recall fundamental notations and definition of periods together with pre-$t$-motives. We also present the relation which shows that KPLs are $q$-th power of HGFs with certain parameters. In Section 2, we consider the deformation of THGFs and KMPLs so as to obtain Theorem \ref{introperiodint} and \eqref{periodkmpl}, the period interpretation of the values of THGFs and KMPLs. In Section 3, we recall refined ABP criterion and and present $(\theta^{q^i}-t)$-expansion of the deformation of THGFs. By using them, we prove Theorem \ref{introlinindthgf}. We deal with linear independence problems among the special values of KMPLs in Section 4 and conclude with our proof for Theorem \ref{intthmkmpl}.

\section{Preliminaries}
\label{No}
\subsection{Notations}

We propose the following symbols.

\begin{itemize}
\setlength{\leftskip}{1.0cm}
\item[$q:=$] a power of the prime number $p$.
\item[$\mathbb{F}_q:=$] a finite field with $q$ elements.
\item[$\theta$, $t:=$] independent variables.
\item[$A:=$] the polynomial ring $\mathbb{F}_q[\theta]$.
\item[$A_{+}:=$] the set of monic polynomials in $A$.
\item[$k:=$] the rational function field $\mathbb{F}_q(\theta)$.
\item[$k_{\infty}:=$] the completion of $k$ at 
infinite place $\infty$, $\mathbb{F}_q((\frac{1}{\theta}))$.
\item[$\overline{k_{\infty}}:=$] a fixed algebraic closure of $k_{\infty}$.
\item[$\mathbb{C}_{\infty}:=$] the completion of $\overline{k_{\infty}}$ at infinity $\infty$.
\item[$\overline{k}:=$] a fixed algebraic closure of $k$ in $\mathbb{C}_{\infty}$.
\item[$|\cdot|_{\infty}:=$] a fixed absolute value for the completed field $\mathbb{C}_{\infty}$ such that $|\theta|_{\infty}=q$.
\item[$\mathbb{T}:=$] the Tate algebra over $\mathbb{C}_{\infty}$, which is the subring of $\mathbb{C}_{\infty}\llbracket t \rrbracket$ that consists of\\
\quad power series convergent on the closed unit disc $|t|_{\infty}\leq 1$.
\item[$\mathbb{E}:=$] $\{\sum_{i=0}^{\infty}a_it^i\in\overline{k}\llbracket t \rrbracket\mid \lim_{i\to\infty}|a_i|_\infty^{1/i}=0,~[k_\infty(a_0,a_1,\dots):k_\infty]<\infty\}.$
\item[$D_i:=$] $\prod^{i-1}_{j=1}([j])^{q^{i-j}}\in A_{+}$ where $[j]:=\theta^{q^j}-\theta$ and $D_0:=1$.
\item[$L_i:=$] $\prod^{i}_{j=1}(-[j])\in A_{+}$ and $L_0:=1$.
\end{itemize}

For $n\in\mathbb{Z}$, we define the following automorphism, which is known as the $n$-fold Frobenius twist:
 	\begin{align*}
	    \mathbb{C}_{\infty}((t))&\rightarrow\mathbb{C}_{\infty}((t))\\
	    f:=\sum_{i}a_it^i&\mapsto \sum_{i}a_i^{q^{n}}t^i=:f^{(n)}.
    \end{align*}

\begin{defn}
For $s\geq 0$, we set $z\in\mathbb{C}_{\infty}$ with $|z|_{\infty}<q^{s}$
and define the following power series:
\[
    \mathcal{Li}_{K,s}(z):=\sum_{i\geq1}\frac{z^{q^i}}{(\theta^{q^i}-t)^s}\in\mathbb{C}_{\infty}\llbracket t,z \rrbracket.
\]
This series is specialized to $Li_{K, s}(\alpha)$ with $t=\theta$ and satisfies the following Frobenius difference equation:
\begin{align}\label{defkplfrobdiff}
     \mathcal{Li}_{K,s}(z)^{(-1)}=\frac{z}{\theta-t}+\mathcal{Li}_{K,s}(z).
\end{align}
\end{defn}

We propose the following relation, which is inspired by the well-known relation for Lerch transcendents and HGFs in the classical case.
\begin{prop}\label{genthgfkpl}
For $m\in\mathbb{Z}_{>0}$ and $z\in\mathbb{C}_{\infty}$ with $|z|_{\infty}<q^{s+m-1}$, we have
\[
	\Bigl(~_{s+1}F_s(1, m, \ldots, m; 1+m, \ldots, 1+m)(z^{q^{-m+1}})\Bigr)^{q^m}=\sum_{i\geq 0}\frac{z^{q^{i+1}}}{[i+m]^{s}}.
\]
\end{prop}
\begin{proof}
By using the relation $(1+m)_i=[i+m]^{q^{-m}}(m)_i$ introduced in \cite[(12)]{T95}, we obtain
\begin{align*}
	&\Bigl(~_{s+1}F_s(1, m, \ldots, m; 1+m, \ldots, 1+m)(z^{q^{-m+1}})\Bigr)^{q^m}\\
	&\quad =\biggl(\sum_{i\geq 0}\frac{D_i(m)_i\cdots(m)_i}{(1+m)_i\cdots(1+m)_iD_i}z^{q^{i-m+1}}\biggr)^{q^m}=\biggl(\sum_{i\geq 0}\frac{1}{[i+m]^{nq^{-m}}}z^{q^{i-m+1}}\biggr)^{q^m}=\sum_{i\geq 0}\frac{z^{q^{i+1}}}{[i+m]^{s}}.
\end{align*}

\end{proof}
When $m=1$, the above proof gives the relation for the KPLs and THGFs, which can be considered the positive characteristic analogue of the classical result, and shows that the following for the HGF $\bigl($$_{s+1}F_s(1, \ldots, 1;2, \ldots, 2)(z)\bigr)$ and the classical polylogarithm $Li_s(z)$:
\[
	z\bigl(~_{s+1}F_s(1, \ldots, 1;2, \ldots, 2)(z)\bigr)=Li_s(z).
\]

The $m=1$ case was first found by \cite{TWYZ} for the $s=1$ and $q-1$ cases and by Nagoya University student Daichi Matsuzuki for the $s>0$ case.
\begin{prop}[Matsuzuki, {\cite[p.154]{TWYZ}}]\label{hgslog}
For $n>0$, we have
\[
    _{s+1}F_s(1, \ldots, 1;2, \ldots, 2)(z)^q=Li_{K,s}(z).
\]
\end{prop}

\subsection{Pre-$t$-motives and periods}

We denote $\overline{k}(t)[\sigma, \sigma^{-1}]$ by the noncommutative $\overline{k}(t)$-algebra generated by $\sigma$ and $\sigma^{-1}$, which is subject to the following relation:
\[
 \sigma f=f^{(-1)}\sigma, \quad f\in\overline{k}(t).
\]

\begin{defn}[{\cite[\S 4.4.1]{P08}}]
The pre-$t$-motive $M$ is a left $\overline{k}(t)[\sigma, \sigma^{-1}]$-module that is finite-dimensional over $\overline{k}(t)$.
\end{defn}
The periods of the pre-$t$-motive are defined as follows.
\begin{defn}\label{defperiod}
Let $M$ be a pre-$t$-motive defined by $\Phi\in\Mat_d(\overline{k}(t))$. If there exists $\Psi\in\Mat_d(\mathbb{T})$ such that 
\[
	\Psi^{(-1)}=\Phi\Psi
\]
and the entries of $\Psi^{-1}$ converge at $t=\theta$, we denote the entries of $\Psi^{-1}|_{t=\theta}$ as the periods of $M$.
\end{defn}

\section{Period interpretations}

We set
\begin{align}\label{defd}
	&\mathbb{D}_{i}:=\prod_{j=1}^i(\theta^{q^j}-t)^{q^{i-j}} \text{\quad for $i>0$ and\ } \mathbb{D}_{i}:=1 \text{\ for $i\leq 0$}, \\
	&\mathbb{L}_{i}:=\prod_{j=1}^i(t-\theta^{q^j}) \text{\quad for $i>0$ and \ } \mathbb{L}_{i}:=1 \text{\ for $i\leq 0$},
\end{align}.
The following symbols can be used to describe the Pochhammer-Thakur symbol at $t=\theta$.

\begin{align*}
\langle a \rangle_n:=\begin{cases}
	   					\mathbb{D}_{n+a-1}^{q^{-(a-1)}}\quad &\text{if $a\geq 1$},\\	
	   					1/\mathbb{L}_{-a-n}^{q^n}\quad &\text{if $0\geq a$ and $-a\geq n$}, \\
	   					0\quad &\text{if $n> -a\geq 0$}.
	   				 \end{cases} 
\end{align*}
By using the above symbol, we can define the deformation series of $_rF_s(z)$ as follows.
\begin{defn}
For $a_1, \ldots, a_r\in\mathbb{Z}$ and $b_1, \ldots, b_s\in\mathbb{Z}_{>0}$, we set
\begin{align}\label{s2defohyper} 
_r\mathcal{F}_s(z):=~_r\mathcal{F}_s(a_1, \ldots, a_r; b_1, \ldots, b_s)(z):=\sum_{n\geq 0}\frac{\langle a_1\rangle_n\cdots \langle a_r\rangle_n}{\langle b_1\rangle_n\cdots \langle b_s\rangle_n\mathbb{D}_n}z^{q^n}\in\mathbb{C}_{\infty}\llbracket t^{q^{-d+1}}, z\rrbracket
\end{align}
where $d=\max_{\substack{1\leq i\leq r\\ 1\leq j\leq s}}\{a_i, b_j\}$.
\end{defn}
In the same way as on page 2, we also assume throughout this paper that for a given $_{r}\mathcal{F}_s(a_1, \ldots, a_r; b_1, \ldots, b_s)(z)$, its parameters satisfy $a_{i}\leq a_{j}$ and $b_{i}\leq b_{j}$ for $i\leq j$ without loss of generality.

The formal power series $~_r\mathcal{F}_s(a_1, \ldots, a_r; b_1, \ldots, b_s)(z)$ is equal to $~_rF_s(a_1, \ldots, a_r; b_1, \ldots, b_s)(z)$ at $t=\theta$. Furthermore, we have
\begin{align}\label{deformhgslog}
_{s+1}\mathcal{F}_s(1, \ldots, 1;2, \ldots, 2)(\alpha)^q=\mathcal{Li}_{K,s}(\alpha)\quad (\alpha\in\overline{k}\ \text{and}\ |\alpha|_{\infty}<q^s).
\end{align}
This can be solved in the same manner as the proof for Proposition \ref{hgslog} with $\langle 2 \rangle_n=\langle 1 \rangle_n(\theta^{q^{n+1}}-t)^{1/q}$.
We can show that the $q$-th power of this power series is a non-zero element of Tate algebra.
\begin{prop}
Let $a_i, b_j\in\mathbb{Z}_{>0}\quad (1\leq i\leq s+1, 1\leq j\leq s)$.
Then, for $\alpha\in\mathbb{C}_{\infty}$ with $|\alpha|_{\infty}<q^{\sum_{j=1}^s(b_j-1)-\sum_{j=1}^{s+1}(a_j-1)}$ and $d\geq\max_{i, j}\{a_i, b_j\}$,
\begin{align}\label{tatehgf}
	_{s+1}\mathcal{F}_s(\alpha)^{q^{d-1}}\in\mathbb{T}\backslash \{0\}.
\end{align}
\end{prop}

\begin{proof}
We set $|t|_{\infty}\leq 1$ and $|\alpha|_{\infty}<q^{\sum_{j=1}^s(b_j-1)-\sum_{j=1}^{s+1}(a_j-1)}$. Then, with $|\langle a \rangle_m^{q^{d-1}}|_{\infty}=|\mathbb{D}_{m+a-1}^{q^{d-a}}|_{\infty}=q^{(m+a-1)q^{m+d-1}}$, we compute the value of each term of $_{s+1}\mathcal{F}_s(\alpha)^{q^{d-1}}$ as follows:
\begin{align*}
	|\frac{\langle a_1\rangle_m^{q^{d-1}}\cdots \langle a_{s+1}\rangle_m^{q^{d-1}}}{\langle b_1\rangle_m^{q^{d-1}}\cdots \langle b_s\rangle_m^{q^{d-1}}\mathbb{D}_m^{q^{d-1}}}\alpha^{q^{m+d-1}}|_{\infty}
	&=(q^{(a_1+m-1+\cdots+a_{s+1}+m-1)-(b_1+m-1+\cdots+b_s+m-1+m)})^{q^{m+d-1}}|\alpha|_{\infty}^{q^{m+d-1}}\\
	&=(q^{\sum_{i=1}^{s+1}(a_i-1)-\sum_{j=1}^{s}(b_j-1)}|\alpha|_{\infty})^{q^{m+d-1}}.
\end{align*}
Therefore, $_{s+1}\mathcal{F}_s(\alpha)^{q^{d-1}}$ with $|t|_{\infty}\leq 1$ converges since $q^{\sum_{j=1}^{s+1}(a_j-1)-\sum_{j=1}^{s}(b_j-1)}|\alpha|_{\infty}$ is less than 1. 
Furthermore, the above computation shows that the largest term of  $_{s+1}\mathcal{F}_s(\alpha)^{q^{d-1}}$ with respect to $|-|_{\infty}$ is $\frac{\langle a_1\rangle_1^{q^{d-1}}\cdots \langle a_{s+1}\rangle_1^{q^{d-1}}}{\langle b_1\rangle_1^{q^{d-1}}\cdots \langle b_s\rangle_1^{q^{d-1}}\mathbb{D}_1^{q^{d-1}}}\alpha^{q^{d+1-1}}$. Thus, $_{s+1}\mathcal{F}_s(\alpha)^{q^{d-1}}$ is not zero.

\end{proof}
Because $|(a)_m^{q^{d-1}}|_{\infty}=|D_{m+a-1}^{q^{d-a}}|_{\infty}=q^{(m+a-1)q^{m+d-1}}$, we can show that the largest term of $_{s+1}F_s(\alpha)^{q^{d-1}}$ is $\frac{(a_1)_1^{q^{d-1}}\cdots (a_{s+1})_1^{q^{d-1}}}{(b_1)_1^{q^{d-1}}\cdots (b_s)_1^{q^{d-1}}\mathbb{D}_1^{q^{d-1}}}\alpha^{q^{d}}$ for $|\alpha|_{\infty}<q^{\sum_{j=1}^s(b_j-1)-\sum_{j=1}^{s+1}(a_j-1)}$ by the same calculation as in the above proof. Then, we obtain
\begin{align}\label{nonvanhgf}
	_{s+1}F_s(a_1, \ldots, a_{s+1}; b_1, \ldots, b_s)(\alpha)\neq 0.
\end{align}

We can also check that $P_{{\bf b},d}$ defined by \eqref{qdef} is a certain entire function as follows:
\begin{prop}\label{entireP}
$P_{{\bf b}, d}\in\mathbb{E}$.
\end{prop}
\begin{proof}
Because $[k_{\infty}(\theta^{\frac{1}{q-1}}):k_{\infty}]<\infty$, it is enough to prove that $P_{{\bf b},d}$ is entire and that $P_{{\bf b}, d}\in\overline{k}\llbracket t \rrbracket$.
Based on the definition \eqref{qdef}, it follows that
\begin{align}\label{qfrob}
P_{{\bf b},d}^{(-1)}&=\prod_{\substack{j=1\\ b_j\geq 2}}^{n}\Bigl\{(\theta-t)^{q^{d-2}}\mathbb{D}_{b_j-2}^{q^{d-b_j}}\Bigr\}P_{{\bf b},d}=\prod_{\substack{j=1\\ b_j\geq 2}}^{n}\biggl\{(\theta-t)^{q^{d-2}}\Bigl((\theta^q-t)^{q^{b_j-3}}\cdots(\theta^{q^{b_j-2}}-t)\Bigr)^{q^{d-b_j}} \biggr\}P_{{\bf b}, d}.
\end{align}
We can expand $\prod_{\substack{j=1\\ b_j\geq 2}}^{n}\biggl\{(\theta-t)^{q^{d-1}}\Bigl((\theta^q-t)^{q^{b_j-3}}\cdots(\theta^{q^{b_j-2}}-t)\Bigr)^{q^{d-b_j}} \biggr\}=\sum_{m=0}^{N}f_mt^m\in A[t]$ and $P_{{\bf b},d}=\sum_{l\geq 0}g_lt^l\in k_{\infty}(\theta^{\frac{1}{q-1}})\llbracket t \rrbracket$. Then, \eqref{qfrob} can be written as $\sum_{l\geq 0}g_l^{(-1)}t^l=\sum_{m=0}^{N}f_mt^m\sum_{l\geq 0}g_lt^l$. By comparing the coefficients, we find that $g_l^{(-1)}=\sum_{\substack{m_1+m_2=l\\ N\geq m_1\geq 0,\ \ m_2\geq 0}}f_{m_1}g_{m_2}$. Thus, $g_l\in\overline{k}$ holds by the induction on $l$. The entireness of $P_{{\bf b}, d}$ follows from the Weierstrass factorization theorem introduced in \cite[Theorem 2.14]{Go96}.
\end{proof}

Let $a_i, b_j\in\mathbb{Z}_{>0}\quad (1\leq i\leq n+1,\ 1\leq j\leq n)$ and
$\alpha\in\overline{k}$ with $|\alpha|_{\infty}<q^{\sum_{j=1}^n(b_j-1)-\sum_{i=1}^{n+1}(a_i-1)}$. We set $M_{{\bf a}, {\bf b}}$ to be the pre-$t$-motive defined by
\begin{align}\label{hgfmotive}
    \Phi_{{\bf a}, {\bf b},d}:=
    \begin{pmatrix}
    \prod_{\substack{j=1}}^s(\theta-t)^{q^{d-1}}\mathbb{D}_{b_j-2}^{q^{d-b_j}} & 0 \\
    \prod_{\substack{j=1}}^{s+1}(\theta-t)^{q^{d-1}}\mathbb{D}_{a_j-2}^{q^{d-a_j}}\alpha^{q^{d-2}} & \prod_{\substack{j=1}}^s(\theta-t)^{q^{d-1}}\mathbb{D}_{b_j-2}^{q^{d-b_j}}  
       \end{pmatrix}\in\Mat_2(\overline{k}[t]).
\end{align}

\begin{thm}\label{periodhgs}
Let $a_i, b_j\in\mathbb{Z}_{>0}\quad (1\leq i\leq n+1,\ 1\leq j\leq n)$.
Then, for $\alpha\in\overline{k}$ with $|\alpha|_{\infty}<q^{\sum_{j=1}^n(b_j-1)-\sum_{i=1}^{n+1}(a_i-1)}$ and $d\geq\max_{i, j}\{a_i, b_j\}$, $_{n+1}F_n(a_1, \ldots, a_{n+1}; b_1, \ldots, b_n)(\alpha)^{q^{d-1}}$ multiplied by $P_{{\bf b},d}^{-1}|_{t=\theta}$ is a period of $M_{{\bf a}, {\bf b},d}$.
\end{thm}
\begin{proof}
According to the definition \eqref{defd}, each element of $\mathbb{D}_i$ satisfies
\begin{align}\label{frobd}
	\mathbb{D}_i^{(-1)}=\begin{cases}
					   	(\theta-t)^{q^{i-1}}\mathbb{D}_{i-1}\ &\text{if $i>0$},\\
					     1\ &\text{if $i\leq 0$}.	
					   \end{cases}
\end{align}
Then, we can obtain the Frobenius difference equation based on \eqref{frobd}:
\begin{align*}
	&\Bigl(~_{s+1}\mathcal{F}_s({\bf a};{\bf b})(\alpha)^{q^{d-1}}\Bigr)^{(-1)}=\biggl(\frac{\langle a_1\rangle_0^{q^{d-1}}\cdots \langle a_{s+1}\rangle_0^{q^{d-1}}}{\langle b_1\rangle_0^{q^{d-1}}\cdots \langle b_{s}\rangle_0^{q^{d-1}}\mathbb{D}_0^{q^{d-1}}}\alpha^{q^{d-1}}+\sum_{m\geq 1}\frac{\langle a_1\rangle_m^{q^{d-1}}\cdots \langle a_{s+1}\rangle_m^{q^{d-1}}}{\langle b_1\rangle_m^{q^{d-1}}\cdots \langle b_{s}\rangle_m^{q^{d-1}}\mathbb{D}_m^{q^{d-1}}}\alpha^{q^{m+d-1}}\biggr)^{(-1)}\\
	&=\frac{\prod_{i=1}^{s+1}\Bigl((\theta-t)^{q^{d-2}}\mathbb{D}_{a_i-2}^{q^{d-a_i}}\Bigr)}{\prod_{j=1}^{s}\Bigl((\theta-t)^{q^{d-2}}\mathbb{D}_{b_j-2}^{q^{d-b_j}}\Bigr)}\alpha^{q^{d-2}}\\
	&+\sum_{m\geq 1}\frac{\Bigl((\theta-t)^{q^{m+a_1-2}}\mathbb{D}_{m+a_1-2}\Bigr)^{q^{d-a_1}}\cdots\Bigl((\theta-t)^{q^{m+a_{s+1}-2}}\mathbb{D}_{m+a_{s+1}-2}\Bigr)^{q^{d-a_{n+1}}}}{\Bigl((\theta-t)^{q^{m+b_1-2}}\mathbb{D}_{m+b_1-2}\Bigr)^{q^{d-b_1}}\cdots \Bigl((\theta-t)^{q^{m+b_{s}-2}}\mathbb{D}_{m+b_{s}-2}\Bigr)^{q^{d-b_{s}}}\Bigl((\theta-t)^{q^{m-1}}\mathbb{D}_{m-1}\Bigr)^{q^{d-1}}}\alpha^{q^{m+d-2}}\\
	&=\frac{\prod_{i=1}^{s+1}\Bigl((\theta-t)^{q^{d-2}}\mathbb{D}_{a_i-2}^{q^{d-a_i}}\Bigr)}{\prod_{j=1}^s\Bigl((\theta-t)^{q^{d-2}}\mathbb{D}_{b_j-2}^{q^{d-b_j}}\Bigr)}\alpha^{q^{d-2}}+\sum_{m\geq 1}\frac{\mathbb{D}_{m+a_1-2}^{q^{d-a_1}}\cdots\mathbb{D}_{m+a_{s+1}-2}^{q^{d-a_{s+1}}}}{\mathbb{D}_{m+b_1-2}^{q^{d-b_1}}\cdots \mathbb{D}_{m+b_{s}-2}^{q^{d-b_{s}}}\mathbb{D}_{m-1}^{q^{d-1}}}\alpha^{q^{m+d-2}}\\
	&=\frac{\prod_{i=1}^{s+1}\Bigl((\theta-t)^{q^{d-2}}\mathbb{D}_{a_i-2}^{q^{d-a_i}}\Bigr)}{\prod_{j=1}^s\Bigl((\theta-t)^{q^{d-2}}\mathbb{D}_{b_j-2}^{q^{d-b_j}}\Bigr)}\alpha^{q^{d-2}}+~_{s+1}\mathcal{F}_s({\bf a};{\bf b})(\alpha)^{q^{d-1}}.
\end{align*}
Thus, we obtain
\begin{align}\label{frobhgf}
\Bigl(~_{s+1}\mathcal{F}_s({\bf a};{\bf b})(\alpha)^{q^{d-1}}\Bigr)^{(-1)}=\frac{\prod_{i=1}^{s+1}\Bigl((\theta-t)^{q^{d-2}}\mathbb{D}_{a_i-2}^{q^{d-a_i}}\Bigr)}{\prod_{j=1}^s\Bigl((\theta-t)^{q^{d-2}}\mathbb{D}_{b_j-2}^{q^{d-b_j}}\Bigr)}\alpha^{q^{d-2}}+~_{s+1}\mathcal{F}_s({\bf a};{\bf b})(\alpha)^{q^{d-1}}.
\end{align}
Finally, we have
\begin{align}\label{qhgsfrob}
	&\Bigl(P_{{\bf b},d}~_{s+1}\mathcal{F}_s({\bf a};{\bf b})(\alpha)^{q^{d-1}}\Bigr)^{(-1)}\\
	&\quad =\prod_{\substack i=1}^{s+1}(\theta-t)^{q^{d-1}}\mathbb{D}_{a_i-2}^{q^{d-a_i}}\alpha^{q^{d-2}}P_{{\bf b},d}+\prod_{j=1}^{s}(\theta-t)^{q^{d-1}}\mathbb{D}_{b_j-2}^{q^{d-b_j}}P_{{\bf b},d}~_{s+1}\mathcal{F}_s({\bf a};{\bf b})(\alpha)^{q^{d-1}}.\nonumber
\end{align}
We define the following matrix:
$
	\Psi:=
    \begin{pmatrix}
    P_{{\bf b},d} & 0 \\
    P_{{\bf b},d}~_{s+1}\mathcal{F}_s({\bf a};{\bf b})(\alpha)^{q^{d-1}} & P_{{\bf b},d}
    \end{pmatrix}
$
which is in $\GL_2(\mathbb{T})$ according to \eqref{tatehgf} and Proposition \ref{entireP}. Therefore, $\Psi^{(-1)}=\Phi_{{\bf a}, {\bf b}}\Psi$ follows based on \eqref{frobhgf}. Accordingly, we can obtain the periods of $M_{{\bf a}, {\bf b},d}$ by
$$
	\Psi^{-1}|_{t=\theta}=
	\begin{pmatrix}
    P_{{\bf b},d}^{-1}|_{t=\theta} & 0 \\
    -P_{{\bf b},d}^{-1}|_{t=\theta}~_{s+1}F_s({\bf a};{\bf b})(\alpha)^{q^{d-1}} & P_{{\bf b},d}^{-1}|_{t=\theta}
    \end{pmatrix}.
$$
\end{proof}

\begin{rem}
In the above theorem, $\Phi_{{\bf a}, {\bf b},d}\in\Mat_2(\overline{k}[t])$ and $\Psi\in\GL_2(\mathbb{T})$ such that $\det(\Phi|_{t=0})\neq 0$ and $\Psi^{(-1)}=\Phi_{{\bf a}, {\bf b},d}\Psi$. Then, we get $\Psi\in\Mat_2({\mathbb{E}})$ by \cite[Proposition 3.1.3]{ABP04}.
\end{rem}

Theorem \ref{periodhgs} also presents a period interpretation of the KPLs with $\tilde{\pi}$.
\begin{eg}\label{kplmotive}
For ${\bf a}=(1, \ldots, 1)$, ${\bf b}=(2, \ldots, 2)$ and $d=2$, the pre-$t$-motive $M_{{\bf a}, {\bf b},d}$ (more precisely, the dual $t$-motive introduced in \cite{ABP04}) is defined by the matrix
$
    \Phi_{{\bf a}, {\bf b},d}=
    \begin{pmatrix}
    (t-\theta)^s & 0 \\
    (-1)^s\alpha & (t-\theta)^s
       \end{pmatrix}.
$
Then, $\Phi_{{\bf a}, {\bf b},d}$ satisfies the relation $\Psi^{(-1)}=\Phi_{{\bf a}, {\bf b},d}\Psi$, where
$
    \Psi:=
    \begin{pmatrix}
    \Omega^s & 0 \\
   \Omega^s\mathcal{L}_{K,s}(\alpha) & \Omega^s
    \end{pmatrix}
$
by using \eqref{defkplfrobdiff} and Theorem \ref{periodhgs}. Thus, we can obtain the periods of $M_{{\bf a}, {\bf b},d}$ as the entries of the following matrix:
\[
\Psi^{-1}|_{t=\theta}= 
	\begin{pmatrix}
    \tilde{\pi}^s & 0 \\
    -\tilde{\pi}^sLi_{K,s}(\alpha) & \tilde{\pi}^s  
    \end{pmatrix}.
\]

\end{eg}

\begin{rem}\label{simperintthgf}
As long as we focus on only the period interpretation of the THGFs, we do not need to consider the power series $P_{{\bf b},d}$. Indeed,
\[
    \Phi'_{{\bf a}, {\bf b},d}:=
    \begin{pmatrix}
    1 & 0 \\
    \frac{\prod_{j=1}^{s+1}(\theta-t)^{q^{d-1}}\mathbb{D}_{a_j-2}^{q^{d-a_j}}}{\prod_{j=1}^s(\theta-t)^{q^{d-1}}\mathbb{D}_{b_j-2}^{q^{d-b_j}}}\alpha^{q^{d-2}} & 1 
       \end{pmatrix}\in\Mat_{2}(\overline{k}(t))
\]
 defines a pre-$t$-motive. According to the equation \eqref{frobhgf}, it satisfies $\Psi^{(-1)}={\Phi'}_{{\bf a}, {\bf b},d}\Psi$ with $
	\Psi:=
    \begin{pmatrix}
    1 & 0 \\
    ~_{s+1}\mathcal{F}_s({\bf a};{\bf b})(\alpha)^{q^{d-1}} & 1
    \end{pmatrix}\in\Mat_2(\mathbb{T}).
$
Thus, we can obtain a simpler period interpretation of the THGFs:
$$
	\Psi^{-1}|_{t=\theta}:=
    \begin{pmatrix}
    1 & 0 \\
    -~_{s+1}F_s({\bf a};{\bf b})(\alpha)^{q^{d-1}} & 1
    \end{pmatrix}.
$$
However, for our proof of the independence/transcendence results, we should assign the representation matrix to be in $\Mat_n(\overline{k}[t])$ to apply Chang's refined ABP criterion (Theorem \ref{changcri}, \cite[Theorem1.2]{C09}). Thus, we modified the interpretation with $P_{{\bf b},d}$ as Theorem \ref{periodhgs}.
\end{rem}

\begin{rem}

The pre-$t$-motive $M_{{\bf a, b},d}$ in Example \ref{kplmotive} was also considered by Angles, Ngo Dac, and Tavares Ribeiro's group and Taelman to develop a counterexample to Taelman's conjecture. See \cite{ANDTR} for more details.
\end{rem}

We can extend Example \ref{kplmotive} to the KMPL case as follows, which is similar to the Carlitz multiple polylogarithm case.
\begin{defn}
Set $\mathfrak{s}:=(s_1, \ldots, s_r)\in\mathbb{Z}_{>0}^r$.
Then, for ${\bf z}=(z_1, \ldots, z_r)\in\mathbb{C}_{\infty}^r$ with $|z_i|_{\infty}<q^{s_i}$,
we define the following power series:
\[
	\mathcal{Li}_{K,\mathfrak{s}}({\bf z}):=\sum_{i_1>i_2>\cdots>i_r>0}\frac{z_1^{q^{i_1}}z_2^{q^{i_2}}\cdots z_r^{q^{i_r}}}{(\theta^{q^{i_1}}-t)^{s_1}(\theta^{q^{i_2}}-t)^{s_2}\cdots (\theta^{q^{i_r}}-t)^{s_r}}
\]
which belong to $\mathbb{T}$ since for $|t|_{\infty}\leq 1$, $|z_1^{q^{i_1}}z_2^{q^{i_2}}\cdots z_r^{q^{i_r}}/(\theta^{q^{i_1}}-t)^{s_1}(\theta^{q^{i_2}}-t)^{s_2}\cdots (\theta^{q^{i_r}}-t)^{s_r}|_{\infty}\rightarrow 0$ as $1\leq i_r<\cdots<i_1\rightarrow \infty$.
The following holds according to the definition of the Frobenius $(-1)$-fold twist and the above series expression:
\begin{align}\label{deformkmpl}
	\mathcal{Li}_{K,\mathfrak{s}}({\bf z})^{(-1)}=\frac{z_r}{(\theta-t)^{s_r}}\mathcal{Li}_{K,(s_1, \ldots, s_{r-1})}(z_1, \ldots, z_{r-1})+\mathcal{Li}_{K,\mathfrak{s}}({\bf z})
\end{align}
\end{defn}
We also define the series
\[
	\mathcal{Li}^{\ast}_{K,\mathfrak{s}}({\bf z}):=\sum_{i_1\geq i_2\geq \cdots\geq i_r>0}\frac{z_1^{q^{i_1}}z_2^{q^{i_2}}\cdots z_r^{q^{i_r}}}{(\theta^{q^{i_1}}-t)^{s_1}(\theta^{q^{i_2}}-t)^{s_2}\cdots (\theta^{q^{i_r}}-t)^{s_r}}\in\mathbb{T}
\]
which is equal to the star-version of the KMPLs, $Li^{\ast}_{K,\mathfrak{s}}({\bf z}):=\sum_{i_1\geq i_2\geq \cdots\geq i_r>0}\frac{z_1^{q^{i_1}}z_2^{q^{i_2}}\cdots z_r^{q^{i_r}}}{[i_1]^{s_1}[i_2]^{s_2}\cdots [i_r]^{s_r}}$ at $t=\theta$.

Then, in the same way as the proof for the star-versions of the Carlitz multiple polylogarithms by \cite{CM19, GN}, we obtain the following equations for $1\leq l \leq j\leq r$ by the inclusion-exclusion principle:
\begin{align}\label{kmplperkey1}
&(-1)^{l}\mathcal{Li}^{\ast}_{K, (s_j, \ldots, s_l)}(\alpha_j, \ldots, \alpha_l)\\
&\quad =\sum_{i=l+1}^j(-1)^{i-1}\mathcal{Li}_{K, (s_l, \ldots, s_{i-1})}(\alpha_l, \ldots, \alpha_{i-1})\mathcal{Li}^{\ast}_{K, (s_j, \ldots, s_i)}(\alpha_j, \ldots, \alpha_i)+(-1)^j\mathcal{Li}_{K, (s_l, \ldots, s_{j})}(\alpha_l, \ldots, \alpha_j),\nonumber\\
\label{kmplperkey2}
&(-1)^{j}\mathcal{Li}^{\ast}_{K, (s_j, \ldots, s_l)}(\alpha_j, \ldots, \alpha_l)\\
&\quad =\sum_{i=l+1}^j(-1)^{i}\mathcal{Li}_{K, (s_i, \ldots, s_j)}(\alpha_i, \ldots, \alpha_j)\mathcal{Li}^{\ast}_{K, (s_{i-1}, \ldots, s_l)}(\alpha_{i-1}, \ldots, \alpha_l)+(-1)^l\mathcal{Li}_{K, (s_l, \ldots, s_{j})}(\alpha_l, \ldots, \alpha_j).\nonumber
\end{align}

Based on \eqref{deformkmpl}, it follows that
\begin{align}\label{periodkmpl}
\Psi^{(-1)}_{\mathfrak{s}, \boldsymbol{\alpha}}=\Phi_{\mathfrak{s}, \boldsymbol{\alpha}}\Psi_{\mathfrak{s}, \boldsymbol{\alpha}}
\end{align}
where
\begin{align*}
    &\Phi_{\mathfrak{s}, \boldsymbol{\alpha}}=\\    
    &\begin{pmatrix}
    (t-\theta)^{w}  					 &      0	           &	   		 & \cdots	      & 0  &		0	  \\
    (-1)^{s_r}\alpha_r(t-\theta)^{w-s_r} &  (t-\theta)^{w}     &             &    \ddots      & \vdots       &   \vdots   \\
    0                                    &  (-1)^{s_{r-1}}\alpha_{r-1}(t-\theta)^{w-s_{r-1}} &       &    \ddots   & \vdots       &    \vdots   \\
    \vdots                               &                   0 &             &    \ddots      &   0   &      \vdots    \\
    \vdots                               & \vdots   		   &             &                & (t-\theta)^{w} & 0 \\
    0                                    & 0     	           &             &                & (-1)^{s_1}\alpha_1(t-\theta)^{w-s_1} & (t-\theta)^{w} 
    \end{pmatrix}\\
    &\in\Mat_{r+1}(\overline{k}[t])
\end{align*}
and
\begin{align*}
    &\Psi_{\mathfrak{s}, \boldsymbol{\alpha}}=\\    
    &\begin{pmatrix}
    \Omega^{w}  					 &      0	           &	   		 & \cdots	      & 0  &		0	  \\
    \Omega^w\mathcal{Li}_{K,s_r}(\alpha_r) &  \Omega^{w}     &             &    \ddots      & \vdots       &   \vdots   \\
    \Omega^w\mathcal{Li}_{K,s_{r-1}, s_r}(\alpha_{r-1}, \alpha_r) &  \Omega^w\mathcal{Li}_{K,s_{r-1}}(\alpha_{r-1}) &       &    \ddots   & \vdots       &    \vdots   \\
    \vdots                               & \vdots              &             &    \ddots      &   0   &      \vdots    \\
    \vdots                               & \vdots   		   &             &                & \Omega^{w} & 0 \\
    \Omega^w\mathcal{Li}_{K,\mathfrak{s}}({\boldsymbol \alpha})  & \Omega^w\mathcal{Li}_{K,s_1, \ldots, s_{r-1}}(\alpha_1, \ldots, \alpha_{r-1}) &        & \cdots   & \Omega^w\mathcal{Li}_{K,s_1}(\alpha_1) & \Omega^{w} 
    \end{pmatrix}\in\GL_{r+1}(\mathbb{T}).
\end{align*}
We remark that $\Psi_{\mathfrak{s}, \boldsymbol{\alpha}}\in\Mat_{r+1}(\mathbb{E})$ by \cite[Proposition 3.1.3]{ABP04}.

By using \eqref{kmplperkey1} and \eqref{kmplperkey2}, $\Psi^{-1}$ can be written as follows:
\begin{align*}
    &\Psi^{-1}_{\mathfrak{s}, \boldsymbol{\alpha}}=\\    
    &\begin{pmatrix}
    \Omega^{-w}  					 &      0	           &	   		 & \cdots	      & 0  &		0	  \\
    -\Omega^{-w}\mathcal{Li}^{\ast}_{K,s_r}(\alpha_r) &  \Omega^{-w}     &             &    \ddots      & \vdots       &   \vdots   \\
    (-1)^2\Omega^{-w}\mathcal{Li}^{\ast}_{K,s_{r}, s_{r-1}}(\alpha_{r}, \alpha_{r-1}) &  -\Omega^{-w}\mathcal{Li}^{\ast}_{K,s_{r-1}}(\alpha_{r-1}) &       &    \ddots   & \vdots       &    \vdots   \\
    \vdots                               & \vdots              &             &    \ddots      &   0   &      \vdots    \\
    \vdots                               & \vdots   		   &             &                & \Omega^{-w} & 0 \\
    (-1)^r\Omega^{-w}\mathcal{Li}^{\ast}_{K,\overleftarrow{\mathfrak{s}}}(\overleftarrow{\boldsymbol \alpha})  & (-1)^{r-1}\Omega^{-w}\mathcal{Li}^{\ast}_{K,s_{r-1}, \ldots, s_{1}}(\alpha_{r-1}, \ldots, \alpha_{1}) &        & \cdots   & -\Omega^{-w}\mathcal{Li}^{\ast}_{K,s_1}(\alpha_1) & \Omega^{-w} 
    \end{pmatrix}\\
    &\in\GL_{r+1}(\mathbb{T}).
\end{align*}
Here, we set $\overleftarrow{\mathfrak{s}}=(s_r, s_{r-1}, \ldots, s_1)$ and $\overleftarrow{\boldsymbol \alpha}=(\alpha_r, \alpha_{r-1}, \ldots, \alpha_1)$. Thus, the periods of the pre-$t$-motive defined by $\Phi_{\mathfrak{s}, \boldsymbol{\alpha}}$ are given as $\tilde{\pi}^{w}$ and $(-1)^{\dep(s_{j}, \ldots, s_{l})}\tilde{\pi}^{w}Li^{\ast}_{K,s_{j}, \ldots, s_{l}}(\alpha_{j}, \ldots, \alpha_{l})\ (1\leq l\leq j\leq r)$.

\section{Linear independence results of the THGFs}
In this section, we discuss the transcendence and linear independence results derived by using a refined version of the Anderson--Brownawell--Papanikolas' linear independence criterion for periods. The original version is given in the following statement.

\begin{thm}[{\cite[Theorem 3.1.1]{ABP04}}]\label{abpcri}
Fix $\Phi\in{\rm Mat}_{d}(\overline{k}[t])$ such that ${\rm det}\Phi=c(t-\theta)^s$ for some $c\in\overline{k}^{\times}$ and some $s\in\mathbb{Z}_{\geq 0}$. Suppose that there exists a vector $\psi\in{\rm Mat}_{d\times 1}(\mathbb{E})$ that satisfies
\[
    \psi^{(-1)}=\Phi\psi.
\]
For every $\rho\in{\rm Mat}_{1\times d}(\overline{k})$ such that $\rho\psi(\theta)=0$, there exists a $P\in{\rm Mat}_{1\times d}(\overline{k}[t])$ such that $P(\theta)=\rho$ and $P\psi=0$.
\end{thm}

By the definition, $\det(\Phi_{{\bf a, b}, d})$ is a polynomial in $\overline{k}[t]$ but generally it can not be written by some powers of $(t-\theta)$ multiplied with a constant $c\in\overline{k}$. Therefore, we need to employ the following refined version of the criterion.
\begin{thm}[{\cite[Theorem 1.2]{C09}}]\label{changcri}
We fix a matrix $\Phi=\Phi(t)\in\Mat_l(\overline{k}[t])$ such that $\det\Phi$ is a polynomial in $t$ that satisfies $\det\Phi(0)\neq 0$. Fix a vector $\psi=[\psi_1(t), \ldots, \psi_l(t)]^{tr}\in\Mat_{l\times 1}(\mathbb{E})$ that satisfies the functional equation $\psi^{(-1)}=\Phi\psi$. Let $\xi\in\overline{k}^{\times}\backslash\overline{\mathbb{F}_q}^{\times}$ satisfy
\[
	\det\Phi(\xi^{(-i)})\neq 0\quad \text{for all $i=1, 2,\ldots$}
\]
Then,
\begin{itemize}
\item[(1)] For every vector $\rho\in\Mat_{1\times l}(\overline{k})$ such that $\rho\psi(\xi)=0$, there exists a vector $P=P(t)\in\Mat_{1\times l}(\overline{k}[t])$ such that $P(\xi)=\rho$ and $P\psi=0$.
\item[(2)] ${\rm tr.deg}_{\overline{k}(t)}\overline{k}(t)(\psi_1(t), \ldots, \psi_l(t))={\rm tr.deg}_{\overline{k}(t)}\overline{k}(t)(\psi_1(\xi), \ldots, \psi_l(\xi))$.
\end{itemize}
\end{thm}

Furthermore, for our proof, we can use the following $(\theta^{q^i}-t)$-expansion of $_{r}\mathcal{F}_s(\alpha)$, which follows from the method described in \cite[p.143]{TWYZ}. We again remark that for a given $_{r}\mathcal{F}_s(a_1, \ldots, a_r; b_1, \ldots, b_s)(\alpha)$, we assume throughout this paper that its parameters satisfy $a_{i}\leq a_{j}$ and $b_{i}\leq b_{j}$ for $i\leq j$ without loss of generality. 
\begin{prop}\label{deformthgfexpn}
For a given $~_r\mathcal{F}_s(\alpha)=~_{r}\mathcal{F}_s(a_1, \ldots, a_r; b_1, \ldots, b_s)(\alpha)$ and $j\in\mathbb{Z}$, we define 
\begin{align*}
&a(j)=r-u+1\quad \text{if $a_{u-1}\leq j\leq a_{u}-1$},\\
&b(j)=s-v+1\quad \text{if $b_{v-1}\leq j\leq b_{v}-1$},\\
&c(j)=a(j)-b(j)
\end{align*}
by setting $b_0=1$, $a_0=b_{-1}=-\infty$ and $a_{r+1}=b_{s+1}=+\infty$.
Then, we have
\begin{align}\label{expansionhgf}
	\bigl(~_r\mathcal{F}_s(\alpha)\bigr)^{q^d}=\sum_{n=0}^{\infty}\biggl(\prod_{m=1}^{n+d-1}(\theta^{q^m}-t)^{c(m-n)q^{n+d-m}}\biggr)\alpha^{q^{n+d}}
\end{align}
where $d=\max\{a_1, \ldots, a_{r}, b_1, \ldots, b_s\}$.
\end{prop}

\begin{rem}\label{cvanish}
For $l\geq\max_{\substack{1\leq i\leq r\\ 1\leq j\leq s}}\{a_i, b_j\}$, $c(l)=a(l)-b(l)=0-0=0$ holds according to the definition. Especially when $r=s+1$, for $l\leq 0$, we again obtain $c(l)=a(l)-b(l)=s+1-(s+1)=0$.
\end{rem}
Later, in our proofs of Theorem \ref{lindepdifferentn} and \ref{lindepdiffalpha}, we need to assume some conditions for $c(j)$ due to the following observation.
\begin{rem}\label{cassum}   
For $N>0$ and each $\alpha^{q^{n+d}}$ with $N>n\geq N-d+1$, its coefficient has a pole or zero at $t=\theta^{q^N}$ with order $c(N-n)q^{n+d-N}$. By the definition, the range of the quantity $c(j)$ depends on the parameters $a_1, \ldots, a_r$ and $b_1, \ldots, b_s$ and we do not have $c(j_1)>c(j_2)$ for $j_1>j_2$ or $j_2>j_1$ in general. Thus for some large enough $r$ and $s$, there may exist distinct $n_1, n_2, \ldots n_l$ with $N>n_1, n_2, \ldots, n_l\geq N-d+1$ so that $c(N-n_1)q^{n_1+d-N}=\cdots=c(N-n_l)q^{n_l+d-N}$. By the expression \eqref{expansionhgf},   
\begin{align}\label{cassumex}
	\bigl(~_r\mathcal{F}_s(\alpha)\bigr)^{q^d}
	&=\sum_{n=N}^{\infty}\biggl(\prod_{m=1}^{n+d-1}(\theta^{q^m}-t)^{c(m-n)q^{n+d-m}}\biggr)\alpha^{q^{n+d}}+\sum_{n=N-d+1}^{N-1}\biggl(\prod_{m=1}^{n+d-1}(\theta^{q^m}-t)^{c(m-n)q^{n+d-m}}\biggr)\alpha^{q^{n+d}}\\
	&\quad +\sum_{n=0}^{N-d}\biggl(\prod_{m=1}^{n+d-1}(\theta^{q^m}-t)^{c(m-n)q^{n+d-m}}\biggr)\alpha^{q^{n+d}}.\nonumber
\end{align}
 By multiplying $(\theta^{q^N}-t)^{c(N-n_1)q^{n_1+d-N}}$ on both side of \eqref{cassumex} and substituting $t=\theta^{q^N}$, we get
  \[
  	\Bigl((\theta^{q^N}-t)^{c(N-n_1)q^{n_1+d-N}}\bigl(~_r\mathcal{F}_s(\alpha)\bigr)^{q^d}\Bigr)|_{t=\theta^{q^N}}=\sum_{i=1}^{l}\prod_{\substack{m=1\\m\neq N}}^{n_i+d-1}(\theta^{q^m}-\theta^{q^N})^{c(m-n_i)q^{n_i+d-m}}\alpha^{q^{n_i+d}}
  \]
  Thus we obtain the $k$-linear combination of $\alpha$ with some distinct powers. If we do not assign conditions to $c(j)$ to make $l$ equals to 1, the combination cause problems in our proofs of Theorem \ref{lindepdifferentn} and \ref{lindepdiffalpha}, in showing contradictions to $k$-linear independence of $\alpha_1, \ldots, \alpha_r\in\overline{k}$ and non-vanishing of $\alpha\in\overline{k}$.  
\end{rem}
\subsection{Applications to the special values of the THGFs}

In this section, we describe the equivalent conditions for the transcendence of the THGFs, which is already given by Thakur et al. in \cite{TWYZ}. We reprove it via the period interpretation of the values of the THGFs and Chang's refined ABP criterion. Furthermore, we show the linear independence of some THGFs at algebraic points, which are specialized to the results of the KPLs. 

Here we repeat our settings on page 2 and 8. We assume throughout this paper that for a given THGF $_{r}F_s(a_1,\ldots, a_{r};b_1, \ldots, b_s)(z)$ $\bigl({\rm resp.}\  _{r}\mathcal{F}_s\ {\rm case } \bigr)$, its parameters satisfy $a_{i}\leq a_{j}$ and $b_{i}\leq b_{j}$ for $i\leq j$ without loss of generality.  

\begin{thm}\label{transthm}
Let $a_i, b_j\in\mathbb{Z}_{>0}\quad (1\leq i\leq s+1,\ 1\leq j\leq s)$ and $\alpha\in\overline{k}^{\times}$ satisfy $|\alpha|_{\infty}<q^{\sum_{j=1}^{s}(b_j-1)-\sum_{i=1}^{s+1}(a_i-1)}$.
Then,
$_{s+1}F_s(a_1,\ldots, a_{s+1};b_1, \ldots, b_s)(\alpha)$ is transcendental over $k$ if and only if $b_j>a_{j+1}$ for some $j$.
\end{thm}
\begin{proof}
First, we prove the transcendence of $_{s+1}F_s(a_1,\ldots, a_{s+1};b_1, \ldots, b_s)(\alpha)$ with $b_j>a_{j+1}$ for some $j$. Conversely, we assume that
$$
f+~_{s+1}F_s(a_1,\ldots, a_{s+1};b_1, \ldots, b_s)(\alpha)=0
$$
for some $f\in\overline{k}^{\times}$. This is equivalent to saying $P_{{\bf b},d}|_{t=\theta}f^{q^d}+P_{{\bf b},d}|_{t=\theta}~_{s+1}F_s(a_1,\ldots, a_{s+1};b_1, \ldots, b_s)(\alpha)^{q^d}=0$. Then, by Theorem \ref{periodhgs} and \ref{changcri}, we can lift this to the relation
\begin{align}\label{transeqn}
(g_1, g_2)
\begin{pmatrix}
	P_{{\bf b},d} \\
	P_{{\bf b},d}~_{s+1}\mathcal{F}_s(a_1,\ldots, a_{s+1};b_1, \ldots, b_s)(\alpha)^{q^d}
\end{pmatrix}
=0
\end{align}
where $g_i(t)\in\overline{k}[t]\ (i=1,2)$ such that $g_1(\theta)=f^{q^d}$ and $g_2(\theta)=1$. This is written without $P_{d, {\bf b}}$ as 
\begin{align}\label{keyeqntr}
	g_1(t)+g_2(t)_{s+1}\mathcal{F}_s(a_1,\ldots, a_{s+1};b_1, \ldots, b_s)(\alpha)^{q^d}=0.
\end{align}
Let $N\in\mathbb{Z}_{>0}$ such that $g_2(\theta^{q^N})\neq 0$.
With the expansion \eqref{expansionhgf}, Remark \ref{cvanish} and changing a variable from $m$ to $l+n$, we have
\begin{align}\label{deformhgfsimple}
\bigl(~_{s+1}\mathcal{F}_s(\alpha)\bigr)^{q^d}&=\sum_{n=0}^{\infty}\biggl(\prod_{m=1}^{n+d-1}(\theta^{q^m}-t)^{c(m-n)q^{n+d-m}}\alpha^{q^{n+d}}\biggr)\\
 											  &=\sum_{n=0}^{\infty}\biggl(\prod_{l=1}^{d-1}(\theta^{q^{l+n}}-t)^{c(l)q^{d-l}}\alpha^{q^{n+d}}\biggr).\nonumber
\end{align}
Because $b_j>a_{j+1}$ for some $s+1>j>0$, there exists $d-1\geq l\geq 1$ such that $b_j>l\geq a_{j+1}$. Then $a_{u}-1\geq m \geq a_{u-1}$ for some $s+3>u\geq j+2$ and $b_v-1\geq l \geq b_{v-1}$ for some $j\geq v\geq 1$. Here, we assume that $b_0=1$, $a_0=b_{-1}=-\infty$ and $a_{s+2}=b_{s+1}=+\infty$ as in Proposition \ref{deformthgfexpn}. Thus, we obtain $c(l)<0$ for these $j$. Indeed, $c(l)=a(l)-b(l)=(s-u+2)-(s-v+1)=v+1-u\leq j+1-u\leq -1$.
We decompose $_{s+1}\mathcal{F}_s(\alpha)^{q^d}$ as follows:
\begin{align*}
	( _{s+1}\mathcal{F}_s(\alpha))^{q^d}&=\sum_{n=N}^{\infty}\biggl(\prod_{l=1}^{d-1}(\theta^{q^{l+n}}-t)^{c(l)q^{d-l}}\alpha^{q^{n+d}}\biggr)
	+\sum_{n=0}^{N-1}\biggl(\prod_{l=1}^{d-1}(\theta^{q^{l+n}}-t)^{c(l)q^{d-l}}\alpha^{q^{n+d}}\biggr)\\
	&=\sum_{n=N}^{\infty}\biggl(\prod_{l=1}^{d-1}(\theta^{q^{l+n}}-t)^{c(l)q^{d-l}}\alpha^{q^{n+d}}\biggr)+\sum_{n=0}^{N-d}\biggl(\prod_{l=1}^{d-1}(\theta^{q^{l+n}}-t)^{c(l)q^{d-l}}\alpha^{q^{n+d}}\biggr)\\
	&\quad +\sum_{n=N+1-d}^{N-1}\biggl(\prod_{l=1}^{d-1}(\theta^{q^{l+n}}-t)^{c(l)q^{d-l}}\alpha^{q^{n+d}}\biggr)\\
	\intertext{then, we denote all $d-1\geq l\geq 1$ such that $c(l)<0$ by $l_i\ (i=1, \ldots, r)$ and decompose as}
	&\hspace{-2cm}=\sum_{n=N}^{\infty}\biggl(\prod_{l=1}^{d-1}(\theta^{q^{l+n}}-t)^{c(l)q^{d-l}}\alpha^{q^{n+d}}\biggr)+\sum_{n=0}^{N-d}\biggl(\prod_{l=1}^{d-1}(\theta^{q^{l+n}}-t)^{c(l)q^{d-l}}\alpha^{q^{n+d}}\biggr)\\
	&\hspace{-2cm}\quad +\sum_{\substack{n=N+1-d\\n\neq N-l_1, \ldots, N-l_r}}^{N-1}\biggl(\prod_{l=1}^{d-1}(\theta^{q^{l+n}}-t)^{c(l)q^{d-l}}\alpha^{q^{n+d}}\biggr)+\sum_{n= N-l_1, \ldots, N-l_r}\biggl(\prod_{l=1}^{d-1}(\theta^{q^{l+n}}-t)^{c(l)q^{d-l}}\alpha^{q^{n+d}}\biggr)
\end{align*}
Thus, \eqref{keyeqntr} can be rewritten as
\begin{align*}
&g_1(t)+g_2(t)\sum_{n=N}^{\infty}\biggl(\prod_{l=1}^{d-1}(\theta^{q^{l+n}}-t)^{c(l)q^{d-l}}\alpha^{q^{n+d}}\biggr)\\
	&\quad +g_2(t)\sum_{n=0}^{N-d}\biggl(\prod_{l=1}^{d-1}(\theta^{q^{l+n}}-t)^{c(l)q^{d-l}}\alpha^{q^{n+d}}\biggr)+g_2(t)\sum_{\substack{n=N+1-d\\n\neq N-l_1, \ldots, N-l_r}}^{N-1}\biggl(\prod_{l=1}^{d-1}(\theta^{q^{l+n}}-t)^{c(l)q^{d-l}}\alpha^{q^{n+d}}\biggr)\\
	&=g_2(t)\sum_{n= N-l_1, \ldots, N-l_r}\biggl(\prod_{l=1}^{d-1}(\theta^{q^{l+n}}-t)^{c(l)q^{d-l}}\alpha^{q^{n+d}}\biggr).
\end{align*}
At $t=\theta^{q^N}$, the left-hand side of the above equation is regular, while $\sum_{n= N-l_1, \ldots, N-l_r}\biggl(\prod_{l=1}^{d-1}(\theta^{q^{l+n}}-t)^{c(l)q^{d-l}}\alpha^{q^{n+d}}\biggr)$
(this sum is nonzero because the largest term with respect to $|\cdot|_{\infty}$ is $\prod_{l=1}^{d-1}(\theta^{q^{l+N-l_h}}-t)^{c(l)q^{d-l}}\alpha^{q^{N-l_h+d}}$ where $l_h=\min\{l_1, \ldots, l_r\}$
) on the right-hand side has a pole. Indeed, on the left-hand side, the 1st term $g_1(t)$ is a polynomial, the 2nd sum $\sum_{n=N}^{\infty}\biggl(\prod_{l=1}^{d-1}(\theta^{q^{l+n}}-t)^{c(l)q^{d-l}}\alpha^{q^{n+d}}\biggr)$ is $\bigl(~_{s+1}F_s(\alpha)^{q^d}\bigr)^{q^N}$ at $t=\theta^{q^N}$, and $c(l)$ for $l=N-n$ are not negative in the 3rd sum $\sum_{n=0}^{N-d}\biggl(\prod_{l=1}^{d-1}(\theta^{q^{l+n}}-t)^{c(l)q^{d-l}}\alpha^{q^{n+d}}\biggr)$ and 4th sum $\sum_{\substack{n=N+1-d\\n\neq N-l_1, \ldots, N-l_r}}^{N-1}\biggl(\prod_{l=1}^{d-1}(\theta^{q^{l+n}}-t)^{c(l)q^{d-l}}\alpha^{q^{n+d}}\biggr)$.
Thus, $g_2(t)$ should have a zero at $t=\theta^{q^N}$, and we obtain a contradiction.

Next, we prove that $_{s+1}F_s(\alpha)$ is algebraic when $b_j\leq a_{j+1}$ for all $j$. Due to the former part of the proof, in this case, $c(l)\geq 0$ for any $l$. Indeed, if there exists $i>0$ such that $c(i)<0$, we obtain $a(i)-b(i)=(s+1-u+1)-(s-v+1)=1-u+v<0$ for some $1\leq u\leq s+1$, $1\leq v\leq s$. Then, we have $b_{v-1}\leq i\leq b_{v}-1$ and $a_{u-1}\leq i\leq a_{u}-1$. However, this contradicts $1-u+v<0$, that is, $b_v\leq a_{u-1}$.

By using the expression \eqref{deformhgfsimple},
we have 
\begin{align*}
	\Bigl(~_{s+1}\mathcal{F}_s({\bf a};{\bf b})(\alpha)^{q^{d}}\Bigr)^{(-1)}&=\Bigl(\sum_{n\geq 0}\prod^{d-1}_{l=1}(\theta^{q^{n+d}}-t)^{c(l)q^{d-l}}\alpha^{q^{n+d}}\Bigr)^{(-1)}\\
	&=\prod_{l=1}^{d-1}(\theta^{q^{d-1}}-t)^{c(l)q^{d-l}}\alpha^{q^{d-1}}+~_{s+1}\mathcal{F}_s({\bf a};{\bf b})(\alpha)^{q^{d-1}}.
\end{align*}
Thus, by setting $\Phi_{{\bf a, b}, d}=\begin{pmatrix}
    1 & 0 \\
    \prod_{l=1}^{d-1}(\theta^{q^{d-1}}-t)^{c(l)q^{d-l}}\alpha^{q^{d-1}} & 1
    \end{pmatrix}\in\Mat_2(\overline{k}[t])
$ and $
	\psi=
    \begin{pmatrix}
    1  \\
    ~_{s+1}\mathcal{F}_s({\bf a};{\bf b})(\alpha)^{q^{d}} 
    \end{pmatrix}\in\Mat_{2\times 1}(\mathbb{T}),
$ we have $\psi^{(-1)}=\Phi_{{\bf a, b},d}\psi$ and then, $\psi\in\Mat_{2\times 1}(\mathbb{E})$ by \cite[Proposition 3.1.3]{ABP04}. 
This allows us to apply Theorem \ref{changcri} without $P_{{\bf b},d}$.

By expanding $\prod^{d-1}_{j=1}(\theta^{q^{n+d}}-t^{q^{d-j}})^{c(j)}\alpha^{q^{n+d}}$, we obtain the finite $\mathbb{F}_q[t]$-linear combination $\sum_{H\geq h\geq 1}f_h(t)\theta^{hq^{n+d}}\alpha^{q^{n+d}}$ with $f_h(t)\in\mathbb{F}_q[t]$, and thus we can write
\begin{align}\label{keyeqnalg1}
_{s+1}\mathcal{F}_s(\alpha)^{q^d}=\sum_{n\geq 0}\sum_{H\geq h\geq 1}f_h(t)\theta^{hq^{n+d}}\alpha^{q^{n+d}}=\sum_{H\geq h\geq 1}f_h(t)\sum_{n\geq 0}\theta^{hq^{n+d}}\alpha^{q^{n+d}}.
\end{align}

Then we have the algebraic relation $(\sum_{n\geq 0}\theta^{hq^{n+d}}\alpha^{q^{n+d}})^q=\sum_{n\geq 0}\theta^{hq^{n+d}}\alpha^{q^{n+d}}-\theta^{hq^d}\alpha^{q^d}$ which implies that the sum $\sum_{n\geq 0}\theta^{hq^{n+d}}\alpha^{q^{n+d}}$ is in $\overline{k}$. Thus the expression \eqref{keyeqnalg1} shows that $_{s+1}\mathcal{F}_s(\alpha)^{q^d}\in\overline{k}[t]$ and Theorem \ref{changcri} (2) yields
\[
	0={\rm tr.deg}_{\overline{k}(t)}\overline{k}(t)\bigl\{1,\ _{s+1}\mathcal{F}_s({\bf a};{\bf b})(\alpha)^{q^{d}} \bigr\}={\rm tr.deg}_{\overline{k}}\overline{k}\bigl\{1,\ _{s+1}F_s({\bf a};{\bf b})(\alpha)^{q^{d}} \bigr\}.
\]
Therefore, $_{s+1}F_s({\bf a};{\bf b})(\alpha)$ is algebraic over $k$.
\end{proof}

In the following, we denote all $m$ satisfying $d_s\geq m\geq 0$ by $m_i\ (i=1, \ldots, n)$ where $d_h=\max_{\substack{ 1\leq i\leq h+1\\ 1\leq j\leq h}}\{a_i, b_j\}$ for $h=1, \ldots, s$.
\begin{thm}\label{lindepdifferentn}
For any ${\bf a}_s=(a_{1}, \ldots, a_{s+1})\in\mathbb{Z}_{>0}^{s+1}$, ${\bf b}_s=(b_{1}, \ldots, b_{s})\in\mathbb{Z}_{>0}^{s}$ such that $b_1>a_{s+1}$, we take $n\geq r\geq 1$ such that $b_1-1\geq m_r$. Let ${\bf a}_h=(a_{1}, \ldots, a_{h+1})$, ${\bf b}_h=(b_{1}, \ldots, b_{h})\ (h=1, \ldots, s)$ and $\alpha_h\in\overline{k}^{\times}$ with $|\alpha_h|_{\infty}<q^{\sum_{j=1}^h(b_j-1)-\sum_{i=1}^{h+1}(a_i-1)}$. If $\min_{1\leq i\leq n, i\neq r}\{c(m_i)q^{d-m_i}\}>c(m_r)q^{d-m_r}$, then $_{h+1}F_{h}({\bf a}_h; {\bf b}_h)(\alpha_h)\ (1\leq h\leq s)$ are $\overline{k}$-linearly independent.
\end{thm}
\begin{proof}
Conversely, we assume that $f_0+{f_1}~_{2}F_{1}({\bf a}_1; {\bf b}_1)(\alpha_1)^{q^{d_s}}+\cdots +{f_s}~_{s+1}F_{s}({\bf a}_s; {\bf b}_s)(\alpha_s)^{q^{d_s}}=0$ for some $f_i\in\overline{k}\ (i=0, 1, \ldots, s)$ are not zero.
We consider
$
\Phi=\begin{pmatrix}  \Phi_1 &         &        &  \\
			                 & \Phi_2  &        &  \\ 
	   						 &         & \ddots &  \\
     						 &		   &		& \Phi_s
     \end{pmatrix}
$
and 
$\psi=\begin{pmatrix} \psi_1 \\
					  \psi_2\\
					  \vdots \\
					  \psi_s
		\end{pmatrix}
$
where $\Phi_h=\prod_{\substack{j=h+1}}^s(\theta-t)^{q^{d_s-1}}\mathbb{D}_{b_j-2}^{q^{d_s-b_j}}\Phi_{{\bf a}_h, {\bf b}_h, d_s}$
and $\psi_{h}=
\begin{pmatrix} P_{{\bf b}_s, d_s} \\
				P_{{\bf b}_s, d_s}~_{h+1}\mathcal{F}_{h}({\bf a}_h; {\bf b}_h)(\alpha_h)^{q^{d_s}}
\end{pmatrix}\in\Mat_{2\times 1}(\mathbb{E}).
$
  According to Theorem \ref{periodhgs}, $\psi_{h}^{(-1)}=\Phi_{h}\psi_{h}$ is true for each $h$; thus we have $\psi^{(-1)}=\Phi\psi$.
Then, by using Theorem \ref{changcri}, we have the following $\overline{k}[t]$-linear relation for $g_i(t)\in\overline{k}[t]$ such that $g_i(\theta)=f_i\quad (i=1, \ldots, s)$:
\begin{align}\label{keyhgfgon}
	P_{{\bf b}_s, d_s}{g_0(t)}+P_{{\bf b}_s, d_s}{g_1(t)}~_{2}\mathcal{F}_{1}({\bf a}_1; {\bf b}_1)(\alpha_1)^{q^d}+\cdots +P_{d_s, {\bf b}_s}{g_s(t)}~_{s+1}\mathcal{F}_{s}({\bf a}_s; {\bf b}_s)(\alpha_s)^{q^d}=0.
\end{align}
We can rewrite the above as follows by using \eqref{expansionhgf}:
\begin{align*}
	P_{{\bf b}_s, d_s}\biggl\{g_0(t)+g_1(t)\biggl(\sum_{n=0}^{\infty}\Bigl(\prod_{l=1}^{d_1-1}(\theta^{q^{l+n}}-t)^{c_1(l)q^{d-l}}&\alpha_1^{q^{n+d}}\Bigr)\biggr)+\\
	&\cdots+g_s(t)\biggl(\sum_{n=0}^{\infty}\Bigl(\prod_{l=1}^{d_s-1}(\theta^{q^{l+n}}-t)^{c_s(l)q^{d-l}}\alpha_s^{q^{n+d}}\Bigr)\biggr)\biggr\}=0. 
\end{align*}
Here each $c_h(\cdot)$ is associated to $_{h+1}\mathcal{F}_{h}({\bf a}_h; {\bf b}_h)(\alpha_h)$ and thus $c_s(\cdot)$ is nothing but $c(\cdot)$.

We note that $c_h(m_r)=-h$ for each $h=1, \ldots, s$. Indeed, $b_1-1\geq m_r\geq a_{s+1}$ implies $b_1-1\geq m_r\geq a_{h+1}$ and then, $c_h(m_r)=a_h(m_r
)-b_h(m_r)=h+1-(h+2)+1-(h-1+1)=-h$. 

We set $N$ to be a positive integer such that $g_s(\theta^{q^N})\neq 0$. Then by multiplying $\bigl(P_{{\bf b}_s, d_s}(\theta^{q^N}-t)^{c_s(m_a)q^{d-m_a}}\bigr)^{-1}$ on both sides of \eqref{keyhgfgon} and by the condition $c(m_i)q^{d-m_i}>c(m_r)q^{d-m_r}$ for $i\neq r$, we have
\begin{align*}
	&(\theta^{q^N}-t)^{-c_s(m_r)q^{d-m_r}}g_0(t)+(\theta^{q^N}-t)^{-c_s(m_a)q^{d-m_r}}g_1(t)\biggl(\sum_{n=0}^{\infty}\biggl(\prod_{l=1}^{d_1-1}(\theta^{q^{l+n}}-t)^{c_1(l)q^{d-l}}\alpha_1^{q^{n+d}}\biggr)\biggr)+\\
	&\quad\cdots+(\theta^{q^N}-t)^{-c_s(m_r)q^{d-m_r}}g_{s-1}(t)\biggl(\sum_{n=0}^{\infty}\biggl(\prod_{l=1}^{d_1-1}(\theta^{q^{l+n}}-t)^{c_{s-1}(l)q^{d-l}}\alpha_1^{q^{n+d}}\biggr)\biggr)\\
	&\qquad +(\theta^{q^N}-t)^{-c_s(m_r)q^{d-m_r}}g_s(t)\biggl(\sum_{\substack{n=0\\n\neq N-m_r}}^{\infty}\biggl(\prod_{\substack{l=1}}^{d_s-1}(\theta^{q^{l+n}}-t)^{c_s(l)q^{d-l}}\alpha_s^{q^{n+d}}\biggr)\biggr)\\
	&\qquad +g_s(t)\prod_{\substack{l=1\\l\neq m_j}}^{d_s-1}(\theta^{q^{l+N-m_r}}-t)^{c_s(l)q^{d-l}}\alpha_s^{q^{N-m_r+d}}=0. 
\end{align*}

By substituting $t=\theta^{q^N}$ into the above equation, we obtain $g_s(\theta^{q^N})\alpha_1^{q^N}=0$. Therefore, we obtain a contradiction $\alpha_1\neq 0$ and then, the desired result holds.
\end{proof}

\begin{rem}\label{kplgon}
For a given $s>0$, if we specialize ${\bf a}_s=(1, \ldots, 1)\in\mathbb{Z}_{>0}^{s+1}$ and ${\bf b}_s=(2, \ldots, 2)\in\mathbb{Z}_{>0}^{s}$ in the above theorem, it shows that $1, Li_{K, s}(\alpha_s), Li_{K, s-1}(\alpha_{s-1}), \ldots, Li_{K, 1}(\alpha_1)$ are $\overline{k}$-linearly independent. Accordingly, there are no $\overline{k}$-linear relations among 1 and the KPLs with different weights. 
\end{rem}

In the following, we denote all $m$ satisfying $d\geq m\geq 0$ by $m_i\ (i=1, \ldots, n)$.
\begin{thm}\label{lindepdiffalpha}
 For any ${\bf a}=(a_{1}, \ldots a_{s+1})\in\mathbb{Z}_{>0}^{s+1}$ and ${\bf b}=(b_{1}, \ldots b_{s})\in\mathbb{Z}_{>0}^{s}$ such that $b_j>a_{j+1}$ for some $j$ and $\min_{1\leq i\leq n, i\neq u}\{c(m_i)q^{d-m_i}\}>c(m_u)q^{d-m_u}$ for some $u$, let $\alpha_i\in\overline{k}^{\times}\ (i=1,\ldots, r)$ with $|\alpha_i|_{\infty}<q^{\sum_{j=1}^s(b_j-1)-\sum_{i=1}^{s+1}(a_i-1)}$.
If $\alpha_1 \ldots, \alpha_r$ are $k$-linearly independent, then $_{s+1}F_{s}({\bf a}; {\bf b})(\alpha_1), \ldots, _{s+1}F_{s}({\bf a}; {\bf b})(\alpha_r)$ are $\overline{k}$-linearly independent.
\end{thm}
\begin{proof}
We assume on the contrary that there exists a nontrivial $\overline{k}$-linear relation:
$$
	{f_1}~_{s+1}F_{s}({\bf a}; {\bf b})(\alpha_1)^{q^d}+\cdots+{f_r}~_{s+1}F_{s}({\bf a}; {\bf b})(\alpha_r)^{q^d}=0.		
$$
We define the matrices $\Phi$ and $\psi$ as
$$
\Phi=\begin{pmatrix}
      \prod_{j=1}^s(\theta-t)^{q^{d-1}}\mathbb{D}_{b_j-2}^{q^{d-b_j}} &  & & \\
      \prod_{j=1}^{s+1}(\theta-t)^{q^{d-1}}\mathbb{D}_{a_i-2}^{q^{d-a_j}}\alpha_1^{q^{d-2}} & \prod_{j=1}^s(\theta-t)^{q^{d-1}}\mathbb{D}_{b_j-2}^{q^{d-b_j}} &  & \\
      \vdots  &  & \ddots & \\
      \prod_{\substack{j=1}}^{s+1}(\theta-t)^{q^{d-1}}\mathbb{D}_{a_j-2}^{q^{d-a_j}}\alpha_r^{q^{d-2}} & & & \prod_{j=1}^s(\theta-t)^{q^{d-1}}\mathbb{D}_{b_j-2}^{q^{d-b_j}}
     \end{pmatrix}\in\Mat_{r+1}(\overline{k}[t])
$$
and $
\psi=\begin{pmatrix}
    P_{{\bf b}, d}  \\
    P_{{\bf b},d}~_{s+1}\mathcal{F}_{s}({\bf a}; {\bf b})(\alpha_1)^{q^d}  \\
     \vdots  \\
    P_{{\bf b},d}~_{s+1}\mathcal{F}_{s}({\bf a}; {\bf b})(\alpha_r)^{q^d}
     \end{pmatrix}\in\Mat_{(r+1)\times 1}(\mathbb{E}).
$
According to Theorem \ref{periodhgs}, $\psi^{(-1)}=\Phi\psi$; then, we can apply Theorem \ref{changcri} and obtain the following:
\begin{align}\label{keyeqndiffalpha}
	g_1(t)~_{s+1}\mathcal{F}_{s}({\bf a}; {\bf b})(\alpha_1)^{q^d}+\cdots+g_r(t)~_{s+1}\mathcal{F}_{s}({\bf a}; {\bf b})(\alpha_r)^{q^d}=0
\end{align}
where $g_i(t)\in\overline{k}[t]$ with $g_i(\theta)=f_i$. We assume $g_r(t)\neq 0$ without loss of generality and set $g_i'(t)=g_i(t)/g_r(t)$. Next, we transform \eqref{keyeqndiffalpha}. We divide both sides of \eqref{keyeqndiffalpha} by $g_r(t)$ and obtain
\begin{align}\label{keyeqndiffalpha1}
	g_1'(t)~_{s+1}\mathcal{F}_{s}({\bf a}; {\bf b})(\alpha_1)^{q^d}+\cdots+g_r'(t)~_{s+1}\mathcal{F}_{s}({\bf a}; {\bf b})(\alpha_r)^{q^d}=0
\end{align}
Then, based on Theorem \ref{frobhgf} and \eqref{deformhgfsimple}, the $(-1)$-fold Frobenius twist of \eqref{keyeqndiffalpha1} is
\begin{align}\label{keyeqndiffalpha2}
	&g_1'(t)^{(-1)}~_{s+1}\mathcal{F}_{s}({\bf a}; {\bf b})(\alpha_1)^{q^d}+\cdots+g_r'(t)^{(-1)}~_{s+1}\mathcal{F}_{s}({\bf a}; {\bf b})(\alpha_r)^{q^d}\\
	&\quad +g_1'(t)^{(-1)}\prod_{j=1}^{d-1}(\theta^{q^j}-t)^{c(j)q^{d-j}}\alpha_1^{q^{d}}+\cdots+g_{r}'(t)^{(-1)}\prod_{j=1}^{d-1}(\theta^{q^j}-t)^{c(j)q^{d-j}}\alpha_r^{q^{d}}=0\nonumber
\end{align}
According to \eqref{keyeqndiffalpha1} and \eqref{keyeqndiffalpha2}, we have the following:
\begin{align*}
	&h_1(t)~_{s+1}\mathcal{F}_{s}({\bf a}; {\bf b})(\alpha_1)^{q^d}+\cdots+h_{r-1}(t)~_{s+1}\mathcal{F}_{s}({\bf a}; {\bf b})(\alpha_{r-1})^{q^d}+R=0
\end{align*}
where
\begin{align*}
&h_i(t):=g_i'(t)-g_i'(t)^{(-1)}\ \text{and}\\ 
&R:=-\prod_{m=1}^{d-1}(\theta^{q^m}-t)^{c(m)q^{d-m}}\bigl(g_1'(t)^{(-1)}\alpha_1^{q^{d}}+\cdots+g_{r-1}'(t)^{(-1)}\alpha_{r-1}^{q^{d}}+\alpha_r^{q^{d}}\bigr)
\end{align*}
The $j$-th repetition of this transformation gives the following equation:
\begin{align*}
	h_{1,j}(t)~_{s+1}\mathcal{F}_{s}({\bf a}; {\bf b})(\alpha_1)^{q^d}+\cdots+h_{r-j, j}(t)~_{s+1}\mathcal{F}_{s}({\bf a}; {\bf b})(\alpha_{r-j})^{q^d}+R_j=0
\end{align*}
where
$$
	h_{i.j+1}(t)=\frac{h_{i, j}(t)}{h_{r-j, j}(t)}-\Bigl(\frac{h_{i, j}(t)}{h_{r-j, j}(t)}\Bigr)^{(-1)}\ \text{and}\ R_{j+1}=\frac{R_j}{h_{r-j, j}(t)}-\Bigl(\frac{R_j}{h_{r-j, j}(t)}\Bigr)^{(-1)}
$$ with $h_{i, 1}(t)=h_{i}(t)$, $R_1=R$. We repeat the above until (i) $j=r-1$ or (ii) $j$ is equal to some $j'$ such that $h_{1, j'}=\cdots=h_{r-j', j'}=0$. For the case (i), we have 
\begin{align}\label{case1keyrel}
h_{1,r-1}(t)~_{s+1}\mathcal{F}_{s}({\bf a}; {\bf b})(\alpha_1)^{q^d}+R_{r-1}=0.
\end{align}
 We set $N>0$ such that $h_{1,r-1}(t)$ is nonzero and $R_{r-1}$ is regular at $t=\theta^{q^N}$. After multiplying by $(\theta-t)^{-c(m_u)q^{q^{d-m_u}}}$ and substituting $t=\theta^{q^N}$ on both sides of \eqref{case1keyrel}, we obtain $h_{1,r-1}(\theta^{q^N})\prod_{\substack{d-1\geq j\geq 1\\ j\neq m_u}}(\theta^{N-m_a+j}-\theta^{q^N})^{c(j)q^{d-j}}\alpha_1^{q^{N-m_u+d}}=0$. This contradicts $\alpha_1\neq 0$. For the case (ii), due to our assumption of the minimality of $c(m_u)q^{d-m_u}$, multiplying by $(\theta^{q^N}-t)^{c(m_u)q^{d-m_u}}$ on both sides of
\begin{align*}
	h_{1,j'-1}(t)~_{s+1}\mathcal{F}_{s}({\bf a}; {\bf b})(\alpha_1)^{q^d}+\cdots+h_{r-j'+1,j'-1}(t)~_{s+1}\mathcal{F}_{s}({\bf a}; {\bf b})(\alpha_{r-j'+1})^{q^d}+R_{j'-1}=0
\end{align*}
gives
\begin{align}\label{lastkey}
	h_{1,j'-1}(\theta^{q^N})\biggl(\prod_{\substack{l=1\\l\neq m_a}}^{d-1}(\theta^{q^{l+N-m_u}}-&\theta^{q^N})^{c(l)q^{d-l}}\alpha_1^{q^{n+d}}\biggr)+\cdots\\
	& \cdots+h_{r-j'+1,j'-1}(\theta^{q^N})\biggl(\prod_{\substack{l=1\\l\neq m_u}}^{d-1}(\theta^{q^{l+N-m_u}}-\theta^{q^N})^{c(l)q^{d-l}}\alpha_{r-j'+1}^{q^{n+d}}\biggr)=0. \nonumber
\end{align}
Because $h_{1, j'}=\cdots=h_{r-j', j'}=0$, we obtain $h_{i, j'-1}(t)/h_{r-j'+1, j'-1}(t)\in \mathbb{F}_q(t)$. By taking $q^{n+d}$ th root, \eqref{lastkey} becomes a $k$-linear relation between $\alpha_1, \ldots, \alpha_r$, and we obtain a contradiction.
\end{proof}

\begin{rem}
When ${\bf a}=(1, \ldots, 1)$ and ${\bf b}=(2, \ldots, 2)$, the above theorem shows that if $\alpha_1, \ldots, \alpha_r$ are $k$-linearly independent, $Li_{K, s}(\alpha_1), \ldots, Li_{K, s}(\alpha_r)$ are $\overline{k}$-linearly independent.  
\end{rem}
\subsection{Linear/algebraic independence results of the KMPLs}
As applications of Theorem \ref{changcri} and our period interpretation of the KMPLs in \eqref{periodkmpl}, we discuss some linear/algebraic independence results among the depth 2 KMPLs. Furthermore, we compare the KMPLs with other periods and show that the Kochubei multizeta values do not have Eulerian/zeta-like indices with Kochubei zeta values and Carlitz periods.

For the depth 2 KMPLs, we have the following linear independence result.
\begin{thm}\label{lininddkmpl}
For indices $\mathfrak{s}\in\mathbb{Z}_{>0}^2$ with $\wt(\mathfrak{s})=w$ and $\alpha\in\overline{k}^{\times}$ with $|\alpha|_{\infty}<q^{s_1}$, $Li_{K, \mathfrak{s}}(\alpha)$ are $\overline{k}$-linearly independent.
\end{thm}

\begin{proof}
We assume on the contrary that among $Li_{K, \mathfrak{s}_i}(\alpha)\ (i=1, \ldots, r)$ with $\mathfrak{s}_i=(s_{i1}, s_{i2})$ ($\mathfrak{s}_i\neq\mathfrak{s}_j$ for $i\neq j$), there exists a $\overline{k}$-linear relation:
\begin{align}\label{dep2eqn}
	f_1Li_{K, \mathfrak{s}_1}(\alpha)+\cdots+f_rLi_{K, \mathfrak{s}_r}(\alpha)=0
\end{align}
for some $f_i\in\overline{k}^{\times}$. For
\[
\Phi_i=
\begin{pmatrix}
(t-\theta)^w                     & 0                   & 0 \\
(-1)^{s_{i2}}(t-\theta)^{s_{i1}} & (t-\theta)^w        & 0 \\ 
0                                & (-1)^{s_{i1}}\alpha(t-\theta)^{s_{i2}} & (t-\theta)^{w}
\end{pmatrix}\in\Mat_{3}(\overline{k}[t])
\]
and
\[
\psi_i=
\begin{pmatrix}   
\Omega^w         \\
\Omega^w\mathcal{L}_{K, s_{i1}}(\alpha)  \\ 
\Omega^w\mathcal{L}_{K, \mathfrak{s}_i}(\alpha)     
\end{pmatrix}\in\Mat_{3\times 1}(\mathbb{E}),
\]
we define
$\Phi=
\begin{pmatrix}
\Phi_1 &         &        &  \\
			                 & \Phi_2  &        &  \\ 
	   						 &         & \ddots &  \\
     						 &		   &		& \Phi_r
\end{pmatrix}
$
and
$\psi=
\begin{pmatrix}
\psi_1   \\
\psi_2   \\
\vdots   \\
\psi_r 
\end{pmatrix}.
$
Then, we can apply Theorem \ref{changcri} (1) to \eqref{dep2eqn} and obtain the following $\overline{k}[t]$-linear relation
\begin{align*}
	g_1(t)\mathcal{L}_{K, \mathfrak{s}_1}(\alpha)+\cdots+g_r(t)\mathcal{L}_{K, \mathfrak{s}_r}(\alpha)=0
\end{align*}
with $g_i(t)\in\overline{k}[t]$ such that $g_i(\theta)=f_i$. We set $s=\max\{ s_{ij}\ |\ i=1, \ldots, r\ \text{and}\ j=1,2 \}$. For some $i$, there are indices
$\mathfrak{s}_i=(w-s, s)$ or $(s, w-s)$ with $s\geq w-s$. When $s=w-s$, the equation \eqref{dep2eqn} becomes $f_iLi_{K, \mathfrak{s}_i}(\alpha)=0$; however, this contradicts Theorem \ref{transthm}. When $s\neq w-s$, we have three cases and again get contradictions as follows. 

{\it Case 1}: $\mathfrak{s}_i=(w-s, s)$ for some $i$ and $\mathfrak{s}_j\neq (s, w-s)$ with $j\neq i$, we set $N>0$ such that $g_i(\theta^{q^N})\neq 0$. Then, we can set $i=1$ without loss of generality and obtain
\begin{align*}
&(\theta^{q^N}-t)^s\Bigl(g_1(t)\mathcal{L}_{K, \mathfrak{s}_1}(\alpha)+\cdots+g_r(t)\mathcal{L}_{K, \mathfrak{s}_r}(\alpha) \Bigr)\\
&=(\theta^{q^N}-t)^s\Bigl(g_1(t)\sum_{i_1=N>i_2>0}\frac{\alpha^{q^{i_1}}}{(\theta^{q^{i_1}}-t)^s(\theta^{q^{i_2}}-t)^{w-s}}+g_1(t)\sum_{\substack{i_1>i_2>0\\ i_1\neq N}}\frac{\alpha^{q^{i_1}}}{(\theta^{q^{i_1}}-t)^s(\theta^{q^{i_2}}-t)^{w-s}}\\
&\quad +g_2(t)\mathcal{L}_{K, \mathfrak{s}_2}(\alpha)+\cdots+g_r(t)\mathcal{L}_{K, \mathfrak{s}_r}(\alpha) \Bigr)=0.
\end{align*}
By substituting $t=\theta^{q^N}$, we obtain $g_1(\theta^{q^{N}})\sum_{N>i_2>0}\alpha^{q^N}/(\theta^{q^{i_2}}-\theta^{q^N})^{w-s}=0$. This contradicts the assumption that $g_1(\theta^{q^N})$ and that $\alpha$ and $\sum_{N>i_2>0}\alpha^{q^N}/(\theta^{q^{i_2}}-\theta^{q^N})^{w-s}$ are nonzero. 

{\it Case 2}: if $\mathfrak{s}_i=(s, w-s)$ for some $i$ and $\mathfrak{s}_j\neq (w-s, s)$ with $j\neq i$, in the same way of Case 1, we obtain $g_1(\theta^{q^{N}})\sum_{N>i_2>0}\alpha^{q^N}/(\theta^{q^{i_2}}-\theta^{q^N})^{w-s}=0$ and the same contradiction. 

{\it Case 3}: if $\mathfrak{s}_i=(s, w-s)$ and $\mathfrak{s}_j=(w-s, s)$ for some $i$ and $j$, we can set $i=1, j=2$ without loss of generality and set $N>0$ such that both $g_1(\theta^{q^N}),\ g_2(\theta^{q^N})$ are nonzero. We have
\begin{align*}
&(\theta^{q^N}-t)^s\Bigl(g_1(t)\mathcal{L}_{K, \mathfrak{s}_1}(\alpha)+\cdots+g_r(t)\mathcal{L}_{K, \mathfrak{s}_r}(\alpha) \Bigr)\\
&=(\theta^{q^N}-t)^s\Bigl(g_1(t)\sum_{i_1=N>i_2>0}\frac{\alpha^{q^{i_1}}}{(\theta^{q^{i_1}}-t)^s(\theta^{q^{i_2}}-t)^{w-s}}+g_1(t)\sum_{\substack{i_1>i_2>0\\ i_1\neq N}}\frac{\alpha^{q^{i_1}}}{(\theta^{q^{i_1}}-t)^s(\theta^{q^{i_2}}-t)^{w-s}}\\
&\quad +g_2(t)\sum_{i_1>N=i_2>0}\frac{\alpha^{q^{i_1}}}{(\theta^{q^{i_1}}-t)^{w-s}(\theta^{q^{i_2}}-t)^{s}}+g_2(t)\sum_{\substack{i_1>i_2>0\\ i_2\neq N}}\frac{\alpha^{q^{i_1}}}{(\theta^{q^{i_1}}-t)^{w-s}(\theta^{q^{i_2}}-t)^{s}}\\
&\quad+g_3(t)\mathcal{L}_{K, \mathfrak{s}_3}(\alpha)+\cdots+g_r(t)\mathcal{L}_{K, \mathfrak{s}_r}(\alpha) \Bigr)=0.
\end{align*}
By substituting $t=\theta^{q^N}$ into the above equation, we obtain
\[
	g_1(\theta^{q^N})\sum_{N>i_2>0}\frac{\alpha^{q^N}}{(\theta^{q^{i_2}}-\theta^{q^N})^{w-s}}+g_2(\theta^{q^N})Li_{K,w-s}(\alpha)^{q^N}=0.
\]
This contradicts the transcendence of $Li_{K,w-s}(\alpha)$ shown by Theorem \ref{transthm}. 

Thus, we obtain the desired independence result.
\end{proof}
Based on this theorem, it follows that the dimension of the space generated by KMPLs with depth 2 and weight $w\geq 1$ is $2^{w-1}$.

We introduced the sum-shuffle relation for the KMPLs \eqref{shufflekmpl} in Section 0.2. This allows us to translate $\overline{k}$-algebraic relations to $\overline{k}$-linear relations for KMPLs. By using this relation and refined ABP criterion, we can have algebraic independence of the KMPLs with depths 1 and 2, as the following.
\begin{thm}\label{alginddkmpl}
For $s_2\geq s_1>0$, let $\alpha_1, \alpha_2, \alpha_3\in\overline{k}^{\times}$ with $|\alpha_1|_{\infty}, |\alpha_3|_{\infty}<q^{s_1}$ and $|\alpha_2|_{\infty}<q^{s_2}$. Then $Li_{K, (s_1, s_2)}(\alpha_1, \alpha_2)$ and $Li_{K, s_2}(\alpha_3)$ are algebraically independent over $\overline{k}$.
\end{thm}
\begin{proof}
We assume that $Li_{K, (s_1, s_2)}(\alpha_1, \alpha_2)$ and $Li_{K, s_2}(\alpha_3)$ satisfy the following algebraic relation:
\begin{align}\label{keyalgeqn}
	\sum_{\substack{n\geq i\geq 0\\ m\geq j\geq 0}}f_{ij}X^iY^j=0
\end{align}
where $f_{ij}\in\overline{k}$, $X=Li_{K, (s_1, s_2)}(\alpha_1, \alpha_2)$ and $Y=Li_{K, s_2}(\alpha_3)$. 
According to the sum-shuffle relation \eqref{shufflekmpl}, we have $X^iY^j=Li_{K, (s_1, s_2)}(\alpha_1, \alpha_2)^iLi_{K, s}(\alpha_3)^j=\sum_{r_{ij}\geq h\geq 1}Li_{K, \mathfrak{s}_{{ij,h}}}({\boldsymbol \alpha_{ij,h}})$
for some indices $\mathfrak{s}_{ij,h}$ with $\wt(\mathfrak{s}_{ij,h})=i(s_1+s_2)+js_2$ and coordinates ${\boldsymbol \alpha}_{ij, h}$ whose entries are algebraic points. Without loss of generality, for each $(i,j)$, we can set $\mathfrak{s}_{ij,1}=(is_1+js_2, is_2)$ and $\mathfrak{s}_{ij,2}=(is_1, is_2+js_2)$. Then, it follows that $\dep(\mathfrak{s}_{ij,h})>2$ for $h>2$ by the description of \eqref{shufflekmpl} in introduction. Here we remark that for each $(i,j)$, $\wt(\mathfrak{s}_{ij,1})=\wt(\mathfrak{s}_{ij,2})=\cdots=\wt(\mathfrak{s}_{ij,r_{ij}})$ but $\mathfrak{s}_{ij, 1}\neq \mathfrak{s}_{ij, h}\ (h\neq 1)$ and $\mathfrak{s}_{ij, 2}\neq \mathfrak{s}_{ij, h}\ (h\neq 2)$. 
The algebraic relation \eqref{keyalgeqn} is written as the following linear relation:
\begin{align}\label{keyalgind}
	\sum_{\substack{n\geq i\geq 0\\ m\geq j\geq 0}}f_{ij}\bigl(\sum_{r_{ij}\geq h\geq 1}Li_{K, \mathfrak{s}_{ij,h}}({\boldsymbol \alpha_{ij,h}})\bigr)=0.
\end{align}
Here $r_{00}=1$ and $Li_{K, \mathfrak{s}_{00,1}}=1$.
We set $l=\max\{i+j\ |\ n\geq i\geq 0, m\geq j\geq 0,\ \text{and}\ f_{ij}\neq 0 \}$ and $i'=\max\{ n\geq i\geq 0\ |\ i+j=l,\ m\geq j\geq 0 \}$. Then for $j'=l-i'$, we obtain $\wt(\mathfrak{s}_{i'j',1})=\cdots =\wt(\mathfrak{s}_{i'j', r_{i'j'}})=\max\{\wt(\mathfrak{s}_{ij,h})\ |\ n\geq i\geq 0, m\geq j\geq 0, r_{ij}\geq h\geq 0\}$.

For the square matrices $M_i$, we set the notation $\bigoplus_{n\geq i\geq 1}M_i$ to be $\begin{pmatrix}
M_1 &                &  \\
	   						 &          \ddots &  \\
     						 &		   		& M_n
\end{pmatrix}$ 
and for the column vectors ${\bf v}_i$, we set $\bigoplus_{n\geq i\geq 1}{\bf v}_i$ to be $\begin{pmatrix}
{\bf v}_1  \\
	   						          \vdots  \\
     						 		   		 {\bf v}_n
\end{pmatrix}$. By using these notations, we define the block diagonal matrix $\Phi=\bigoplus_{\substack{n\geq i\geq 0\\ m\geq j\geq 0}}\Phi_{ij}$ and the column vector $\psi=\bigoplus_{\substack{n\geq i\geq 0\\ m\geq j\geq 0}}\psi_{ij}$ 
where $\Phi_{00}=(t-\theta)^{i'(s_1+s_2)+j's_2}\in\overline{k}[t]$, $\psi_{00}=\Omega^{\wt(\mathfrak{s}_{i'j',1})}\in\mathbb{E}$ and for $(i,j)\neq (0,0)$,
$$\Phi_{ij}=(t-\theta)^{\wt(\mathfrak{s}_{i'j',1})-\wt(\mathfrak{s}_{ij,1})}
\begin{pmatrix}
\Phi_{\mathfrak{s}_{ij,1}, {\boldsymbol \alpha_{ij,1}}} &         &        &  \\
			                 & \Phi_{\mathfrak{s}_{ij,2}, {\boldsymbol \alpha_{ij,2}}}  &        &  \\ 
	   						 &         & \ddots &  \\
     						 &		   &		& \Phi_{\mathfrak{s}_{ij, r_{ij}}, {\boldsymbol \alpha_{ij,r_{ij}}}}
\end{pmatrix}\in\Mat_{r_{ij}}(\overline{k}[t])
$$
$$
\psi_{ij}=\Omega^{\wt(\mathfrak{s}_{i'j',1})-\wt(\mathfrak{s}_{ij,1})}
\begin{pmatrix}
\psi_{\mathfrak{s}_{ij,1}, {\boldsymbol \alpha_{ij,1}}}   \\
\psi_{\mathfrak{s}_{ij,2}, {\boldsymbol \alpha_{ij,2}}}   \\
\vdots   \\
\psi_{\mathfrak{s}_{ij,r_{ij}}, {\boldsymbol \alpha_{ij, r_{ij}}}}
\end{pmatrix}\in\Mat_{r_{ij}\times 1}(\mathbb{E}).
$$
By \eqref{periodkmpl}, we have $\Psi^{(-1)}=\Phi\Psi$ and we can apply Theorem \ref{changcri} for \eqref{keyalgind}. Then we obtain the following:
\begin{align}\label{keyalgdeform0}
\Omega^{\wt(\mathfrak{s}_{i'j',1})}\sum_{\substack{n\geq i\geq 0\\ m\geq j\geq 0}}g_{ij}(t)\bigl(\sum_{r_{ij}\geq h\geq 1}\mathcal{L}_{K, \mathfrak{s}_{ij,h}}({\boldsymbol \alpha_{ij,h}})\bigr)=0
\end{align}
where $g_{ij}(t)\in\overline{k}[t]$ such that $g_{ij}(\theta)=f_{ij}$. 
Dividing both sides of equation \eqref{keyalgdeform0} by $\Omega^{\wt(\mathfrak{s}_{i'j',1})}$, we get 
\begin{align}\label{keyalgdeform}
\sum_{\substack{n\geq i\geq 0\\ m\geq j\geq 0}}g_{ij}(t)\bigl(\sum_{r_{ij}\geq h\geq 1}\mathcal{L}_{K, \mathfrak{s}_{ij,h}}({\boldsymbol \alpha_{ij,h}})\bigr)=0
\end{align}

Let $N>0$ such that $g_{ij}(\theta^{q^N})\neq 0$ for all $i, j$. After multiplying by $(\theta^{q^N}-t)^{ls_2}$ on both sides of \eqref{keyalgdeform}, we rewrite the left--hand side by
\begin{align*} 
	&(\theta^{q^N}-t)^{ls_2}\Bigl\{\sum_{\substack{n\geq i\geq 0\\ m\geq j\geq 0\\ i+j=l}}g_{ij}(t)\Bigl(\sum_{r_{ij}\geq h\geq 1}\mathcal{L}_{K, \mathfrak{s}_{ij,h}}({\boldsymbol \alpha_{ij,h}})\Bigr)+\sum_{\substack{n\geq i\geq 0\\ m\geq j\geq 0\\ i+j<l}}g_{ij}(t)\Bigl(\sum_{r_{ij}\geq h\geq 1}\mathcal{L}_{K, \mathfrak{s}_{ij,h}}({\boldsymbol \alpha_{ij,h}})\Bigr)\Bigr\}\\
	&=(\theta^{q^N}-t)^{ls_2}\Bigl\{\sum_{\substack{n\geq i\geq 0\\ m\geq j\geq 0\\ i+j=l}}g_{ij}(t)\Bigl(\mathcal{L}_{K, \mathfrak{s}_{ij,1}}({\boldsymbol \alpha_{ij,1}})+\mathcal{L}_{K, \mathfrak{s}_{ij,2}}({\boldsymbol \alpha_{ij,2}})+\sum_{r_{ij}\geq h\geq 3}\mathcal{L}_{K, \mathfrak{s}_{ij,h}}({\boldsymbol \alpha_{ij,h}})\Bigr)\\
	&\quad +\sum_{\substack{n\geq i\geq 0\\ m\geq j\geq 0\\ i+j<l}}g_{ij}(t)\Bigl(\sum_{r_{ij}\geq h\geq 1}\mathcal{L}_{K, \mathfrak{s}_{ij,h}}({\boldsymbol \alpha_{ij,h}})\Bigr)\Bigr\}\\	
	&=(\theta^{q^N}-t)^{ls_2}\Bigl\{\sum_{\substack{n\geq i\geq 0\\ m\geq j\geq 0\\ i+j=l}}g_{ij}(t)\Bigl(\mathcal{L}_{K, (is_1+js_2, is_2)}(\alpha_1\alpha_3, \alpha_2)+\mathcal{L}_{K, (is_1, is_2+js_2)}(\alpha_1, \alpha_2\alpha_3)+\sum_{r_{ij}\geq h\geq 3}\mathcal{L}_{K, \mathfrak{s}_{ij,h}}({\boldsymbol \alpha_{ij,h}})\Bigr)\\
	&\quad +\sum_{\substack{n\geq i\geq 0\\ m\geq j\geq 0\\ i+j<l}}g_{ij}(t)\Bigl(\sum_{r_{ij}\geq h\geq 1}\mathcal{L}_{K, \mathfrak{s}_{ij,h}}({\boldsymbol \alpha_{ij,h}})\Bigr)\Bigr\}.
\end{align*}
Then we have
\begin{align}\label{keykeyalgind}
&(\theta^{q^N}-t)^{ls_2}\Bigl\{\sum_{\substack{n\geq i\geq 0\\ m\geq j\geq 0\\ i+j=l}}g_{ij}(t)\Bigl(\mathcal{L}_{K, (is_1+js_2, is_2)}(\alpha_1\alpha_3, \alpha_2)+\mathcal{L}_{K, (is_1, is_2+js_2)}(\alpha_1, \alpha_2\alpha_3)+\sum_{r_{ij}\geq h\geq 3}\mathcal{L}_{K, \mathfrak{s}_{ij,h}}({\boldsymbol \alpha_{ij,h}})\Bigr)\\
	&\quad +\sum_{\substack{n\geq i\geq 0\\ m\geq j\geq 0\\ i+j<l}}g_{ij}(t)\Bigl(\sum_{r_{ij}\geq h\geq 1}\mathcal{L}_{K, \mathfrak{s}_{ij,h}}({\boldsymbol \alpha_{ij,h}})\Bigr)\Bigr\}=0.\nonumber
\end{align}
At $t=\theta^{q^N}$, $\mathcal{L}_{K, (is_1+js_2, is_2)}(\alpha_1\alpha_3, \alpha_2)$ (when $s_2=s_1$) and $\mathcal{L}_{K, (is_1, is_2+js_2)}(\alpha_1, \alpha_2\alpha_3)$ with $i+j=l$ have a pole with order $ls_2$. Indeed, we have 
\[
\mathcal{L}_{K, (is_1+js_2, is_2)}(\alpha_1\alpha_3, \alpha_2)=\sum_{\substack{i_1>i_2>0\\i_1\neq N}}\frac{\alpha_1^{q^{i_1}}\alpha_2^{q^{i_2}}}{(\theta^{q^{i_1}}-t)^{is_1+js_2}(\theta^{q^{i_2}}-t)^{is_2}}+\sum_{N=i_1>i_2>0}\frac{\alpha_1^{q^{i_1}}\alpha_2^{q^{i_2}}}{(\theta^{q^{i_1}}-t)^{is_1+js_2}(\theta^{q^{i_2}}-t)^{is_2}}
\] 
and 
\[
\mathcal{L}_{K, (is_1, is_2+js_2)}(\alpha_1, \alpha_2\alpha_3)=\sum_{\substack{i_1>i_2>0\\i_2\neq N}}\frac{\alpha_1^{q^{i_1}}\alpha_2^{q^{i_2}}}{(\theta^{q^{i_1}}-t)^{is_1}(\theta^{q^{i_2}}-t)^{is_2+js_2}}+\sum_{i_1>N=i_2>0}\frac{\alpha_1^{q^{i_1}}\alpha_2^{q^{i_2}}}{(\theta^{q^{i_1}}-t)^{is_1}(\theta^{q^{i_2}}-t)^{is_2+js_2}}.
\]
Then, for $i+j=l$, we obtain $\{(\theta^{q^N}-t)^{ls_2}\mathcal{L}_{K, (is_1, is_2+js_2)}(\alpha_1, \alpha_2\alpha_3)\}|_{t=\theta^{q^N}}=\bigl(\alpha_2\alpha_3Li_{K, is_1}(\alpha_1)\bigr)^{q^N}$ when $s_1<s_2$ and
 $\{(\theta^{q^N}-t)^{ls_2}\mathcal{L}_{K, (is_1+js_2, is_2)}(\alpha_1\alpha_3, \alpha_2)\}|_{t=\theta^{q^N}}=(\alpha_1\alpha_3)^{q^N}\sum_{N>i_2>0}\frac{\alpha_2^{q^{i_2}}}{(\theta^{q^{i_2}}-\theta^{q^N})^{is_2}}$ when $s_1=s_2$.    
On the other hand, the other terms in the equation \eqref{keykeyalgind} have a pole with an order which is less than $ls_2$. Because the other $\mathcal{L}_{K, \mathfrak{s}_{ij,h}}$ satisfy the case $i+j<l$ and $0<h$, or the case $i+j=l$ and $2<h$. Here we remind that when $h>2$, $\dep(\mathfrak{s}_{ij,h})>2$ as shown in the beginning of our proof. 
Therefore by substituting $t=\theta^{q^N}$ in \eqref{keykeyalgind}, we get 
\[
	0=\begin{cases}
		&\sum_{\substack{n\geq i\geq 0\\ m\geq j\geq 0\\ i+j=l}}g_{ij}(\theta^{q^N})\bigl(\alpha_2\alpha_3Li_{K, is_1}(\alpha_1)\bigr)^{q^N}\quad \text{if $s_1<s_2$},\\
		&\sum_{\substack{n\geq i\geq 0\\ m\geq j\geq 0\\ i+j=l}}g_{ij}(\theta^{q^N})\Bigl\{(\alpha_1\alpha_3)^{q^N}\sum_{N>i_2>0}\frac{\alpha_2^{q^{i_2}}}{(\theta^{q^{i_2}}-\theta^{q^N})^{is_2}}+\bigl(\alpha_2\alpha_3 Li_{K, is_1}(\alpha_1)\bigr)^{q^N} \Bigr\} \quad \text{if $s_1=s_2$}.
	  \end{cases}	
\]
By taking the $q^N$-th root of the above two equations, we obtain $\overline{k}$-linear relations among 1 and the KPLs with different weights in both the $s_2>s_1$ and $s_2=s_1$ cases. This contradicts Remark \ref{kplgon}, which follows from Theorem \ref{lindepdifferentn}. Therefore we get the desired result.
\end{proof}

Furthermore, we can consider the linear independence of the KMPLs with other periods. For example, we can compare Carlitz polylogarithms and KPLs as follows.

\begin{thm}\label{linindkplcpl}
Let $\alpha, \beta\in\overline{k}^{\times}$ such that $|\alpha|_{\infty}<q^n$, $|\beta|_{\infty}\in q^{nq/(q-1)}$. Then $Li_{K,n}(\alpha), Li_{C,n}(\beta)$ are $\overline{k}$-linearly independent.
\end{thm}
\begin{proof}
For the pre-$t$-motive defined by $\Phi=\begin{pmatrix}  (t-\theta)^n & 0            & 0 \\
															(-1)^n\alpha     & (t-\theta)^n & 0 \\ 
															\beta^{(-1)}(t-\theta)^n & 0 & 1 \end{pmatrix}\in\Mat_{3}(\overline{k}[t])$,
we have $\psi^{(-1)}=\Phi\psi$ for $\psi=\begin{pmatrix} \Omega^n \\
														 \Omega^n\mathcal{Li}_{K,n}(\alpha)\\
														 \Omega^n\mathcal{Li}_{C,n}(\beta)  \end{pmatrix}\in\Mat_{3\times 1}(\mathbb{E})$
where $\mathcal{Li}_{C,n}(\beta)=\beta+\sum_{i>0}\beta^{q^i}/\mathbb{L}_{i}^{n}$ that satisfies $\mathcal{Li}_{C,n}(\beta)^{(-1)}=\beta^{(-1)}+\mathcal{Li}_{C,n}(\beta)/(t-\theta)^{n}$.														 														
We assume on the contrary that $f_1Li_{K,n}(\alpha)+f_2Li_{C,n}(\beta)=0$ for some $f_i\in\overline{k}^{\times}\ (i=1,2)$. Then, we have $\tilde{\pi}^{-n}\Bigl(g_1Li_{K,n}(\alpha)+g_2Li_{C,n}(\beta)\Bigr)=0$. Based on Theorem \ref{abpcri}, this relation can be extended as follows:
\begin{align}\label{keyklcl}
	g_1(t)\Omega^n\mathcal{Li}_{K, n}(\alpha)+g_2(t)\Omega^n\mathcal{Li}_{C,n}(\beta)=0.
\end{align}
For some $g_i(t)\in\overline{k}[t]$ with $g_i(\theta)=f_i\ (i=1,2)$. Let $N\in\mathbb{Z}_{>0}$ such that $g_1(\theta^{q^N})\neq 0$. After multiplying by $\mathbb{L}_N^n\Omega^{-n}$ on both sides, we can rewrite \eqref{keyklcl} as follows:
\begin{align}\label{keyklclln}
	&g_1(t)\biggl\{\mathbb{L}_N^n\sum_{i=1}^{N-1}\frac{\alpha^{q^i}}{(\theta^{q^i}-t)^n}+(-1)^n\mathbb{L}_{N-1}^n\alpha^{q^N}+\mathbb{L}_N^n\sum_{i>N}\frac{\alpha^{q^i}}{(\theta^{q^i}-t)^n}\biggr\}\\
	&+g_2(t)\biggl\{(t-\theta^q)^n\cdots(t-\theta^{q^N})^n\beta+(t-\theta^{q^2})\cdots(t-\theta^{q^N})^n\beta^{q}\cdots+(t-\theta^{q^N})^n\beta^{q^{N-1}}\nonumber\\
	&+\sum_{i>N}\frac{\beta^{q^i}}{(t-\theta^{q^{N+1}})^n\cdots(t-\theta^{q^i})^n}\biggr\}=0.\nonumber	
\end{align}
We have $\sum_{i>N}\frac{\beta^{q^i}}{(t-\theta^{q^{N+1}})^n\cdots(t-\theta^{q^i})^n}|_{t=\theta^{q^N}}=Li_{C,n}(\alpha)^{q^N}-\alpha^{q^N}$. Therefore, by substituting $t=\theta^{q^N}$, the equation \eqref{keyklclln} becomes
\[
	g_1(\theta^{q^N})\bigr((\theta^{q^N}-\theta^q)^n\cdots(\theta^{q^N}-\theta^{q^{N-1}})^n\bigr)\alpha^{q^N}+g_2(\theta^{q^N})Li_{C,n}(\alpha)^{q^N}-g_2(\theta^{q^N})\beta^{q^N}=0.
\]
This forces $Li_{C,n}(\beta)=g_2(\theta^{q^N})^{-1/q^N}\{g_1(\theta^{q^N})\beta^{q^N}-g_2(\theta^{q^N})\}^{1/q^N}\in\overline{k}$ while $Li_{C,n}(\beta)$ is transcendental over $k$ by \cite[Theorem 5.4.3]{C14}. Therefore, we obtain a contradiction and the desired $\overline{k}$-linear independence result.
\end{proof}

Next, we compare the KMPLs, Carlitz period and KPLs as the following two theorems.

\begin{thm}\label{noneulerian}
Let $\mathfrak{s}=(s_{1}, \ldots, s_{r})\in\mathbb{Z}_{>0}^r$ with $wt(\mathfrak{s})=w$ and ${\boldsymbol \alpha}=(\alpha_1, \ldots, \alpha_r)\in(\overline{k}^{\times})^r$ with $|\alpha_i|_{\infty}<q^{s_i}\ (i=1, \ldots, r) $. Then, $Li_{K,\mathfrak{s}}({\boldsymbol \alpha})$ and $\tilde{\pi}^w$ are $\overline{k}$-linearly independent.
\end{thm}
\begin{proof}
We assume on the contrary that there exists a nontrivial $\overline{k}$-linear equation $f_1\tilde{\pi}^w+f_2Li_{K, \mathfrak{s}}({\boldsymbol \alpha})=0$. With \eqref{periodkmpl}, we have $\psi^{(-1)}=\Phi\psi$, where

\begin{align*}
    \Phi=\begin{pmatrix}
    \Phi_{\mathfrak{s}, \boldsymbol{\alpha}} & \\
    										 & 1
    \end{pmatrix}\in\Mat_{r+2}(\overline{k}[t]), \quad \text{and}\quad
\psi=\begin{pmatrix}
		\Omega^{w}  					     \\
    \Omega^w\mathcal{Li}_{K,(s_r)}(\alpha_r)   \\
    \Omega^w\mathcal{Li}_{K,(s_{r-1}, s_r)}(\alpha_{r-1}, \alpha_r)     \\
    \vdots                                  \\
    \vdots                                  \\
    \Omega^w\mathcal{Li}_{K,\mathfrak{s}}({\boldsymbol \alpha})  \\
    1 
	 \end{pmatrix}\in\Mat_{r+2\times 1}(\mathbb{E}).
\end{align*}
Then, by using Theorem \ref{abpcri}, we can obtain the following $\overline{k}[t]$-linear relation
\begin{align}\label{keyeulerian}
	g_1(t)\Omega^w\mathcal{Li}_{K,\mathfrak{s}}({\boldsymbol \alpha})+g_2(t)\cdot 1=0
\end{align}	
for some $g_i(t)\in\overline{k}[t]\ (i=1,2)$ such that $g_i(\theta)=f_i$. By taking the $(-1)$-fold Frobenius twist of \eqref{keyeulerian}, we obtain
\begin{align}\label{keyeuleriantwisted}
	&g_1(t)^{(-1)}(-1)^{s_r}(t-\theta)^{w-s_r}\Omega^w\alpha_r\mathcal{Li}_{K,(s_1, \ldots, s_{r-1})}(\alpha_1, \ldots, \alpha_{r-1})+g_1(t)^{(-1)}(t-\theta)^w\Omega^w\mathcal{Li}_{K,\mathfrak{s}}({\boldsymbol \alpha})\\
	&+g_2(t)^{(-1)}=0.\nonumber
\end{align}
We set $N\in\mathbb{Z}_{>0}$ such that $g_2(t)^{(-1)}$ is nonzero at $t=\theta^{q^N}$. By definition, $\Omega^w$ has a zero at $t=\theta^{q^N}$ with an order $w$, while the terms of both $\mathcal{Li}_{K,(s_1, \ldots, s_{r-1})}(z_1, \ldots, z_{r-1})$ and $\mathcal{Li}_{K,\mathfrak{s}}({\bf z})$ have poles at $t=\theta^{q^N}$ with orders strictly less than $w$. Thus, $\Omega^w\mathcal{Li}_{K,(s_1, \ldots, s_{r-1})}(z_1, \ldots, z_{r-1})$ and $\Omega^w\mathcal{Li}_{K,\mathfrak{s}}({\bf z})$ vanish at $t=\theta^{q^N}$. Then, by substituting $t=\theta^{q^N}$ on both sides of \eqref{keyeuleriantwisted}, we obtain the contradiction $g_2(t)^{(-1)}|_{t=\theta^{q^N}}=0$. Therefore, we obtain the desired result.
\end{proof}

\begin{thm}\label{nonzetalike}
Let $\mathfrak{s}=(s_{1}, \ldots, s_{r})\in\mathbb{Z}_{>0}^r$
such that $wt(\mathfrak{s})=w$. For $\boldsymbol{\alpha}=(\alpha_1, \ldots, \alpha_r)$ with $|\alpha_i|_{\infty}<q^{s_i}$ and $|\beta|_{\infty}<q^{w}$, $Li_{K,\mathfrak{s}}(\boldsymbol{\alpha})$ and $Li_{K, w}(\beta)$ are $\overline{k}$-linearly independent.
\end{thm}
\begin{proof}
We assume on the contrary that there exists a nontrivial $\overline{k}$-linear relation for some $f_i\in\overline{k}\ (i=1,2)$:
\[
	f_1Li_{K, \mathfrak{s}}(\boldsymbol{\alpha})+f_2Li_{K, w}(\beta)=0.
\]
We define the matrices
$
    \Phi=\begin{pmatrix}
    \Phi_{\mathfrak{s}, \boldsymbol{\alpha}} & \\
    	\boldsymbol{\beta}					& (t-\theta)^w
    \end{pmatrix}\in\Mat_{r+2}(\overline{k}[t])
$
where $\boldsymbol{\beta}=(\beta, 0, \ldots, 0)\in\Mat_{1\times r}(\overline{k})$ and
\begin{align*}
\psi=\begin{pmatrix}
		\Omega^{w}  					     \\
    \Omega^w\mathcal{Li}_{K,s_r}(\alpha_r)    \\
    \Omega^w\mathcal{Li}_{K,s_{r-1}, s_r}(\alpha_{r-1}, \alpha_r)     \\
    \vdots                                  \\
    \vdots                                  \\
    \Omega^w\mathcal{Li}_{K,\mathfrak{s}}({\boldsymbol \alpha})  \\
    \Omega^w\mathcal{Li}_{K,w}(\beta) 
	 \end{pmatrix}\in\Mat_{r+2\times 1}(\mathbb{E}).
\end{align*}
By using Theorem \ref{abpcri}, we obtain a $\overline{k}$-linear equation
$g_1(t)\Omega^w\mathcal{L}_{K, \mathfrak{s}}(\boldsymbol{\alpha})+g_2(t)\Omega^w\mathcal{L}_{K,w}(\beta)=0 $
and then
\[
		g_1(t)\mathcal{L}_{K, \mathfrak{s}}(\boldsymbol{\alpha})+g_2(t)\mathcal{L}_{K,w}(\beta)=0.
\]
By taking the $(-1)$-fold Frobenius twist, we obtain
\begin{align}\label{keyeqnzetalike}
	&g_1(t)^{(-1)}(\theta-t)^{-s_r}\alpha_r\mathcal{L}_{K,(s_1,\ldots,s_{r-1})}(\alpha_1, \ldots, \alpha_{r-1})+g_1(t)^{(-1)}\mathcal{L}_{K,\mathfrak{s}}(\boldsymbol{\alpha})\\
	&\hspace{6.5cm} +g_2(t)^{(-1)}(\theta-t)^{-w}\beta+g_2(t)^{(-1)}\mathcal{L}_{K,w}(\beta)=0.\nonumber
\end{align}
We set $N>0$ such that $g_2(t)^{(-1)}|_{t=\theta^{q^N}}\neq 0$. After multiplying by $(\theta^{q^N}-t)^w$ on both sides of \eqref{keyeqnzetalike}, we obtain
\begin{align*}
	&(\theta^{q^N}-t)^w\Bigl(g_1(t)^{(-1)}(\theta-t)^{-s_r}\alpha_n\mathcal{L}_{K,(s_1,\ldots,s_{r-1})}(\alpha_1, \ldots, \alpha_{r-1})+g_1(t)^{(-1)}\mathcal{L}_{K,\mathfrak{s}}(\boldsymbol{\alpha})\Bigr)\\
	&\hspace{0cm} +(\theta^{q^N}-t)^wg_2(t)^{(-1)}(\theta-t)^{-w}\beta+(\theta^{q^N}-t)^wg_2(t)^{(-1)}\sum_{\substack{i>0\\ i\neq N}}\frac{\beta^{q^i}}{(\theta^{q^i}-t)^w}+g_2(t)^{(-1)}\beta^{q^N}=0.\nonumber
\end{align*}
By substituting $t=\theta^{q^N}$, we obtain a relation $g_2(t)^{(-1)}|_{t=\theta^{q^N}}\beta^{q^N}=0$. This contradicts $g_2(t)^{(-1)}|_{t=\theta^{q^N}}\neq 0$ and $\beta\neq 0$. Therefore we obtain the desired result.
\end{proof}
\section*{Acknowledgements}
The author is grateful to Chieh-Yu Chang and Dinesh Thakur for giving him fruitful comments. Especially he is greatly indebted to Chieh-Yu Chang for suggesting him to the topic of this paper. He is also thankful to Yen-Tsung Chen, Og\u{u}z Gezm\.{i}\c{s}, Daichi Matuzuki and Changningphaabi Namoijam for helpful comments. The author wishes to thank National Center for Theoretical Sciences in Hsinchu where the most part of this paper was written, for the hospitality and JSPS Overseas Research Fellowships for their financial support. This work is also supported by JSPS KAKENHI Grant Number JP22J00006.

\end{document}